\documentclass[11 pt]{amsart}

[1995/06/01]% LaTeX date must be June 1995 or later

\usepackage{amssymb,amsmath}
\usepackage{graphicx}
\usepackage{mathdots}
\usepackage{amsbsy}
\usepackage{amscd}
\usepackage{amsthm}
\usepackage{mathrsfs}
\usepackage{verbatim}
\usepackage[colorlinks]{hyperref}
\usepackage{authblk}
\usepackage{fullpage}
\usepackage[english]{babel}
\usepackage[T1]{fontenc}
\usepackage[utf8]{inputenc}
\usepackage{enumitem}

\input xy
\usepackage[all]{xy}

\usepackage{url}
\usepackage[english]{babel}
\usepackage[T1]{fontenc}
\usepackage[utf8]{inputenc}
\usepackage{booktabs}
\usepackage{blkarray}
\usepackage[fleqn,tbtags]{mathtools}
\usepackage{graphics,graphicx}
\usepackage{pstricks,pst-node,pst-tree}

\usepackage{tikz}
\usetikzlibrary{backgrounds,fit,decorations.pathreplacing}
\usetikzlibrary{arrows,shapes,positioning}
\usetikzlibrary{calc,through,backgrounds,chains}
\usetikzlibrary{snakes,automata,backgrounds,petri}
\usetikzlibrary{arrows.meta}
\usetikzlibrary{shapes,decorations,arrows,calc,arrows.meta,fit,positioning}
\usepackage{calc}
\usetikzlibrary{decorations.markings}
\tikzstyle{vertex}=[circle, draw, inner sep=0pt, minimum size=6pt]
\theoremstyle{definition}
\newtheorem{Theorem}{Theorem}[section]
\newtheorem{Lemma}[Theorem]{Lemma}
\newtheorem{Example}[Theorem]{Example}
\newtheorem{Remark}[Theorem]{Remark}
\newtheorem{Definition}[Theorem]{Definition}
\newtheorem{Corollary}[Theorem]{Corollary}
\newtheorem{Proposition}[Theorem]{Proposition}
\newtheorem*{Theorem*}{Theorem}

\newcommand{\C}{\mathbb{C}}
\newcommand{\N}{\mathbb{N}}

\newcommand{\Z}{\mathbb{Z}}

\newcommand{\mc}[1]{\mathcal{#1}} % short for mathcal
\definecolor {processblue}{cmyk}{0.96,0,0,0}

\begin{document}

\mathtoolsset{showonlyrefs}

\title{Stability of the Hecke algebra of wreath products}

\author{\c{S}afak \"Ozden\footnote{Address: Weigandufer 7, 12045, Berlin. email: sozden@tulane.edu.}}
\maketitle

{\centering \c{S}afak \"Ozden\footnote{Address: Weigandufer 7, 12045, Berlin. email: sozden@tulane.edu.}\par }

{\centering\footnotesize In memory of Cemal Koç\footnote{ http://journals.tubitak.gov.tr/math/issues/mat-10-34-4/mat-34-4-1-1009-43.pdf}\par}

\begin{abstract}
The Hecke algebras $\mathcal{H}_{n,k}$ of the group pairs $(S_{kn}, S_k\wr S_n)$ can be endowed with a filtration with respect to  
the orbit structures of the elements of $S_{kn}$ relative to the action of $S_{kn}$ on the set of $k$-partitions
of $\{1,\dots,kn\}$. We
prove that the structure constants of the associated filtered algebra $\mathcal{F}_{n,k}
$ is independent of $n$. The stability property enables the construction of a universal algebra
$\mc{F}$ to govern the algebras $\mathcal{F}_{n,k}$. We also prove that the
structure constants of the algebras $\mathcal{H}_{n,k}$ are polynomials in $n$. For $k=2$, when the algebras $(\mathcal{F}_{n,2})_{n\in \mathbb{N}}$ are commutative, these results were obtained in \cite{AC}, \cite{saf17} and \cite{tout14}.\\ \\
%\textbf{Keywords:} non-abelian Farahat-Higman rings, structure constants.
\end{abstract}

%\tableofcontents

\section{Introduction}

\subsection{} The asymptotic study of sub-algebras of the group algebras $\C[S_n]$  goes back
to the work of Farahat and Higman \cite{FH}. They study the centers
$\mc{Z}_n=\mc{Z}(\C[S_n])$ of the group algebras $\C[S_n]$. To prove Nakamura's conjecture, they
prove a stability result for the structure coefficients of the algebras
$\mc{Z}_n$. From the perspective of  \cite{Mac} and \cite{WW18}, we
can describe the method of Farahat and Higman as
follows:
\begin{enumerate}
\item Construct a conjugacy invariant numerical function on the symmetric group to measure the complexity of a given element.
\item Endow the centers $\mc{Z}_n$ of the group algebras $\C[S_n]$ with a suitable
filtration invariant under conjugation.
\item (Stability property) Prove
that the structure constants of the associated
filtered algebra $\mc{Z}_n$ are independent of $n$, hence obtain a universal algebra $\mc{Z}$ that governs all the algebras $\mc{Z}_n$.
\end{enumerate}
The term \textit{stability} is introduced in
\cite{WW18}. Wang used this strategy and proved that the family $(G\wr
S_n)_{n\in \N}$, with $G$ finite, satisfies the stability property, \cite{W04}. Recently,
Wan and Wang applied the same strategy to
the family $(GL_n(q))_{n\in \N}$, \cite{WW18}.
Using the filtration induced by the reflection
length they proved that the family $(GL_n(q))_{n\in
\N}$ also satisfies the stability property. Following
\cite{WW18}, we proved that the family
$(Sp_{2n}(q))$ satisfies the stability
property with respect to the reflection length induced from $GL_{2n}(q)$, \cite{safthesis}.

%%%%%%%%%%%%%%%%%%%%%%%%%%%%%%%%%%%%%%%%%%%%%%%%%%%%%%%%%%%%%%%%%%%%%%%%%%%%%%%%%%%%%%%%%%%% SECOND PARAGRAPH %%%%%%%%%%%%%%%%%%%%%%%%%%%%%%%%%%%%%%%%%%%%%%%%%%%%%%%%%%%%%%%%%%%%%%%%%%%%%%%%%%%%%%%%%%%%

%%%%%%%%%%%%%%%%%%%%%%%%%%%%%%%%%%%%%%%%%%%%%%%%%%%%%%%%%%%%%%%%%%%%%%%%%%%%%%%%%%%%%%%%%%%% 3RD PARAGRAPH %%%%%%%%%%%%%%%%%%%%%%%%%%%%%%%%%%%%%%%%%%%%%%%%%%%%%%%%%%%%%%%%%%%%%%%%%%%%%%%%%%%%%%%%%%%%

Farahat and Higman's approach produces a master algebra which maps surjectively onto individual algebras, but there are no direct relation between the individual algebras. Molev and Olshanski construct the individual algebras out of partial permutations, with natural maps between the consecutive algebras. Consequently, the master algebra is constructed as a projective limit. Using partial isomorphism method, Ivanov
and Kerov
calculated the structure constants of $\mc{Z}_n$ and
proved that the family $(S_n)_{n\in \N}$ satisfies
the stability property, \cite{ivanov2001algebra}. M{\'e}liot also used the partial isomorphism method, \cite{PM14}, and
proved that the family $GL_n(q)$ satisfies the
stability property. Likewise, Tout proved that the structure constants of the family $S_k\wr S_n$ are polynomials in $n$, \cite{tout17}.

%%%%%%%%%%%%%%%%%%%%%%%%%%%%%%%%%%%%%%%%%%%%%%%%%%%%%%%%%%%%%%%%%%%%%%%%%%%%%%%%%%%%%%%%%%%% PARAGRAPH 5 %%%%%%%%%%%%%%%%%%%%%%%%%%%%%%%%%%%%%%%%%%%%%%%%%%%%%%%%%%%%%%%%%%%%%%%%%%%%%%%%%%%%%%%%%%%%

\subsection{}  The central algebra $\mc{Z}(\C[G])$ can be viewed as the Hecke algebra of group pair, which enables one to generalize the previous work to a wider a class of algebras. In fact, recall that when $G$ is a finite group and $H$ its a subgroup, then the Hecke algebra $\mathcal{H}(G,H)$
attached to the pair $(G,H)$ is defined as the sub-algebra of the complex valued functions on $G$
that are invariant along $H$-double cosets. If two $\C$-valued functions $\phi$ and $\psi
$ on $G$ are contained in the $\mathcal{H}(G,H)$ then their multiplication $\phi\cdot \psi$ is defined by the rule

\begin{equation}\label{convolution product}
(\phi \cdot \psi)(x):=\sum_{y\in G}\phi(y)\psi(y^{-1}x)=\sum_{yz=x}\phi(y)\psi(z).
\end{equation}
The algebra $\mc{Z}(\C[G])$ is isomorphic to the Hecke algebra $\mathcal{H}(diag(G)\backslash G\times G/ diag(G)$ where $diag(G)=\{(g,g)\in G\times
G:g\in G\}$, cf. \cite[Th.1.5.22]{CTS10}. As a result, the class of double-coset algebras is more general than the class of the centers of group algebras. One can pose the question
of stability property for the family of double-coset algebras.
In the literature, the only double-coset algebra so far considered is
$(S_{2n},B_n)$, where $B_n$ is the centralizer of
$\tau_n=(1,2)(3,4)\cdots (2n-1, 2n)$. It turns out that the family of the double-coset algebras of the pair $(S_{2n},B_n)$ also satisfies the stability property.
This result was obtained by Aker and Can in
\cite{AC}, and by Can and the author in \cite{saf17}
as generalizations of the Farahat and
Higman method. Tout proved the same stabiltity as a
generalization of the partial isomorphism method of Ivanov-Kerov, \cite{tout14}. We note that the double-coset algebra of the pair $(S_{2n},B_n)$ is commutative.

\subsection{}
We consider the family $(\mc{H}_{n,k})_{n\in \N}$ of
algebras $\mc{H}_{n,k}=\C[(S_k\wr S_n)\backslash S_{kn}/(S_k\wr S_n)]$, where $k$ is a fixed positive integer greater or equal to $2$. These algebras are non-commutative unless $k=2$. We show that there is a natural
filtration on the non-commutative Hecke algebra  
$\mc{H}_{n,k}$, and the structure constants of the induced filtered
algebra $\mc{F}_n$ is independent of $n$. In particular, $\mc{H}_{n,k}$
satisfies the stability property.
Taking $k=2$ we reproduce the case of the pair $(S_{2n},B_n)$ mentioned above. We use the
Farahat-Higman approach but our proof differs
from the one presented in \cite{AC}. The
proof in \cite{AC} relies on a twist of the
automorphism $g\longmapsto g^{\tau_n}$,
which does not have a direct generalization
to the case $k\neq 2$. Instead we focus
on the action of $S_{kn}$ on the set of
$k$-partitions of the set $\{1,2,\cdots,kn\}$
and reduce the proof of the stability
property to proving that the \textit{"orbits"}
of certain elements must intersect \textit{"transversely"}.
In the rest of the
introduction we briefly explain these terms.

The group $S_{kn}$ acts transitively
on the set of
$k$-partitions of the set $[kn]:=\{1,\cdots,kn\}$.
The wreath product $S_k\wr S_n$ can be identified with the stabilizer of the $k$-partition
$$\{1,2,\cdots,k\},\\ \cdots, \{k(n-1)+1,\cdots,kn\}$$
of $[kn]$. For $g\in S_{kn}$, we define a
graph $G_g$ (Def. \ref{graph of an element}), called the \textit{type of} $g$,
on the set of $n$-vertices. The isomorphism class
of the graph $G_g$ completely determines the $S_k\wr
S_n$-double coset of $g\in S_{kn}$. We
define the modified type ${G^{\circ}_g}$
as the graph obtained from $G_g$ by removing
the connected components that consist of a
single vertex. The modified type of an
element does not change under the embedding
$S_{kn}\hookrightarrow S_{k(n+1)}$. The set
of modified types with $n$-vertices is
denoted by $\mc{G}_{n,k}$ and the union of modified types is denoted by $\mc{G}_{\infty,k}$. Let $M\in \mc{G}_{n,k}$
be a modified type. The sizes of the connected components of $M$ define a
partition $\lambda_M$ of $n$. If
$\lambda_M=(\lambda_1,\cdots,\lambda_t)$,
the weight $||M||$ of the graph
$M$ is defined as the integer
$\sum_{i=1}^t(\lambda-1)$. The algebra
$\mc{H}_n$ is generated by the double coset sums $$S_M(n):=\sum_{\substack{g\in S_{kn}\\ {G_g}^{\circ}=M}}g.$$
The multiplication in the algebra $\mc{H}_n$ can be written as
\[
S_M(n)\cdot S_N(n) = \sum_{ ||L||\leq ||M||+||N||} c_{M,N}^L(n) \cdot S_L(n),
\]
where $M,N,L\in \mc{G}_{\infty,k}$. We are ready to state the main result of this paper.

\begin{Theorem}
The functions $c_{M,N}^L(n)$ are polynomials in $n$. If the equality $||L||= ||M||+||N||$ holds then the function $c_{M,N}^L(n)$ is independent of $n$.
\end{Theorem}

%It is worth noting the following.
%In order to establish the stability
%property we introduce a graph theoretical construction, which we call edge replacement, cf. Def. \ref{edge replacement definition}. This construction consists of three graphs $M,N,L$ and a sequence of graphs
%$N=N_0,N_1,\cdots,N_t=L$, called the
%\textit{replacement chain}, where the graphs in the chain are obtained using the connected components of $M$. We observe that, in order to guarantee
%the equality $||L||= ||M||+||N||$, the Young
%diagrams $Y_i$ obtained by the sizes of
%the connected components of $N_i$ should satisfy
% the following relationship, which was already
% brought into attention by Ivanov and Kerov in
% \cite[Corollary 8.3]{ivanov2001algebra}: Let $Y$ denote the Young diagram
%   obtained from the sizes of the connected
%   components of $M$. Then
%   $Y_{i+1}$ is obtained
%from $Y_i$
%\begin{enumerate}
% \item By replacing $l_i$ many rows of $Y_i$ of length $n_1,\cdots,n_{t_i}$
% with a row of length $(\sum_{j=1}^{l_i} n_j)-1$, where $Y=(l_1,\cdots,l_t)$.
% \item Keeping the remaining rows intact.
%\end{enumerate}
%This property of the chain replacement allows one to use the
%same method to recover the stability result of Farahat-Higman by
%considering the graph attached to $x\in S_n$ with the %edge set $\{i,x(i)\}$.
%However, we do not present this proof.

\subsection{}
The paper is organized as follows. In section 2 we
define the group $S_k\wr S_n$ as a subgroup of $S_n$ and introduce the necessary notation. In section 3
we review the basics of the
double-coset algebras and
introduce the abstract set-up
that captures the properties
satisfied by the family
$\C[(S_k\wr S_n)\backslash S_{kn}/(S_k\wr S_n)]$. We also present a
way of calculating the structure constants in terms of
centralizers. In section 4, we parameterize the double cosets of $S_k\wr S_n$,
prove the polynomiality of
the structure constants. In section 5 we define the modified type,
construct the universal algebra
$\mc{F}_{\infty,k}$, and state the
stability theorem. Section 6 is devoted to the graph
theoretic construction graph evolution and some of its properties are proved.
In the final section we prove the
stability theorem.

%%%%%%%%%%%%%%%%%%%%%%%%%%%%%%%%%%%%%%%%%%%%%%%%%%%%%%%%%%%%%%%%%%%%%%%%%%%%%%%%%%%%%%%%%%%%%%%%%%%%%%%%%%%%%%%%%%%%%%%%%%%%%%%%%%%%%%%%%%%%%%%%%%%%%%%%%%%%%%%%%%%%%%%%SECTION 2%%%%%%%%%%%%%%%%%%%%%%%%%%%% %%%%%%%%%%%%%%%%%%%%%%%%%%%%%%%%%%%%%%%%%%%%%%%%%%%%%%%%%%%%%%%%%%%%%%%%%%%%%%%%%%%%%%%%%%%%%%%%%%%%%%%%%%%%%%%%%%%%%%%%%%%%%%%%%%%%%%%%

\section{The group $S_k\wr S_n$:}

In this section we will define the group $S_k\wr S_n$ as the stabilizer of an element with respect to the action of $S_{kn}$ on the set of $k$-partitions of the set $\{1,2,\cdots,kn\}$. Following the definition we will elaborate the consequences of the definition. We will also introduce the Hecke algebra of the pair $(S_{kn},S_k\wr S_n)$, which will be our main interest in the rest of the article. We start with introducing the necessary notation.

The set of positive integers is denoted by $\N_+$. For every positive integer $n$ in $\N_+$, the set $\{1,2,\dots, n \}$ is denoted by $[n]$. If $g$ is a permutation of the set of positive integers $\N_+$, then the \textit{support} of $g$ is defined to be the set of the positive integers $\N(g)=\{r\in \N_+: g(r)\neq r \}$ that the are moved by $g$.  Let $A$ be a subset of positive integers $\N_+$. The set of permutations $g$ of $A$ with finite support $\N(g)$ is denoted by $S_A$. If the subset $A$ is equal to $[n]$ for some positive integer $n$, then by abuse of notation we will stick to the usual notation and write $S_n$ rather than $S_{[n]}$. Likewise, if $A=\N_+$ then we will write $S_{\infty}$ rather than $S_{\N_+}$. Elements of $S_{\infty}$ are called \textit{finitary permutations} of $\N$.

Let $k$ be a fixed positive integer greater or equal to $2$. A \textit{$k$-partition} of an arbitrary set $X$ is a collection of disjoint $k$-elemental subsets of $X$ where the union covers $X$. For every positive integer $i$ we introduce the set
\[
\Gamma_i=\{k(i-1)+1,\cdots, ki\}.
\]
The collection $(\Gamma_{i})_{i\in \N_+}$ is a $k$-partition of positive integers $\N_+$. The sets of the form $\Gamma_i$ will be called $\Gamma$-part. Each positive integer $r$ is contained in a unique $\Gamma$-part. If $r\in \Gamma_i$, we say that the $\Gamma$-part of $r$ is $i$, in which case we will write $p(r)=i$.  Two positive integers $r$ and $s$ will be called {\textit{partners}} if
$p(r)=p(s)$, i.e. if they are contained in the same $\Gamma_i$ for some positive integer $i$.
The subgroup $H_{\infty}$ of $S_{\infty}$ is defined to be group of finitary permutations that preserve the partner relationship. More precisely, a permutation $h$ in $S_{\infty}$ is an element of $H_{\infty}$ if and only if the equality $p(r)=p(s)$ implies $ p(h(r))=p(h(s))$,
for all positive integers $r$ and $s$. Notice that the integer $k$ is not used in the notation. But this will not cause any ambiguity since we will only be dealing with a fixed $k$.

The group $H_{\infty}$ can be described in a more suggestive way. In fact, the group of finitary permutations $S_{\infty}$ acts on the set of $k$-partitions of $\N_+$. The point stabilizer of the $k$- partition $(\Gamma_i)_{i\in\N_+}$ is equal to the group $H_{\infty}$. In other words, the group $H_{\infty}$ consists of elements that permutes $\Gamma$-parts. 
%which induces a permutation of the collection $(\Gamma_i)_{i\in \N_+}$. 
Now we generalize the defintion of $H_{\infty}$ to subsets of $\N_+$. Let $A$ be a subset of positive integers $\N_+$ that it can be covered by $\Gamma$-parts. We define $H_A$ as the group of permutations $g$ satisfying the two conditions: 1. The support $\N(g)$ is contained $A$; 2. The permutation $g$ stabilizes the $k$-partition $(\Gamma_i)_{i\in \N_+}$. In other words,
\begin{equation}\label{generalized H-A}
H_A=H_{\infty}\cap S_A.    
\end{equation}
As before, if $A=[kn]$ then by abuse of notation we write
$H_n$ rather than $H_{[kn]}$.

\begin{Remark}\label{tau-ij remark}
Let $n$ be a fixed positive integer. We will define two subgroups $\mc{S}_n$ and $\mc{Y}_n$ of $H_n$ for which the internal semi-direct product $\mc{Y}_n\rtimes \mc{S}_n$ is equal to $H_n$. We start with the definition of $\mc{S}_n$. For each pair of distinct positive integers $i$ and $j$ that are less than or equal to $n$, we define the permutation $\tau_{ij}$ of $[kn]$ by setting
\begin{eqnarray}
\tau_{ij}=\prod_{r=1}^k (k(i-1)+r , k(j-1)+r),
\end{eqnarray}
where $(k(i-1)+r , k(j-1)+r)$ denotes the transposition that interchanges the positive integers $k(i-1)+r$ and $k(j-1)+r$. The permutation $\tau_{ij}$ interchanges the elements of $\Gamma_i$ and $\Gamma_j$ monotone increasingly, and it acts as identity elsewhere. The subgroup of $S_{nk}$ generated by the permutations $\tau_{ij}$ is denoted by $\mc{S}_n$. The association sending the transposition $(ij)\in S_n$ to the permutation $\tau_{ij}\in \mc{S}_n$ defines an isomorphism between $S_n$ and $\mc{S}_n$, which justifies the notation.
Secondly, we define the group $\mc{Y}_n$ as the Young subgroup
$S_{\Gamma_1}\times \dots \times S_{\Gamma_n}$ of $S_{kn}$. The group $\mc{Y}_n$ is a subgroup of $H_n$. It can be shown easily that $\mc{S}_n$ normalizes $\mc{Y}_n$ and $H_n=\mc{Y}_n\rtimes \mc{S}_n$. The internal semi-direct product decomposition implies that every element $h\in H_n$ can be written as a product
\[
h=h_{\mc{Y}}h_{\mc{S}},
\] where $h_{\mc{Y}}\in \mc{Y}_n$
and $h_{\mc{S}}\in \mc{S}_n$. In general, let $A$ be a subset of positive integers $\N_+$ which is equal to union of $\Gamma_j$, $j\in J$, for some subset $J \subset \N$. Discussing as above, one can show that there is an isomorphism $H_A\simeq H_{|J|}=S_k\wr S_{J}$. If $J$ is finite then the cardianlity of $H_A$ is equal to $|J|!(k!)^{|J|}$.
\end{Remark}

\begin{Example}\label{example of elements}
Let $k=3$ and $n=7$. The permutation $g=(1,8,18,21,6,10,13,2,11,3,12)(4,16,19)\\(5,17,20)$ is an element of $S_{21}$. The positive integers $1$ and $2$ are contained in $\Gamma_1=\{1,2,3\}$ and thus $p(1)=p(2)=1$. On the other hand $p(g(1))=p(8)=3$ as $8\in \Gamma_3=\{7,8,9\}$, and $p(g(2))=p(11)=4$. It follows that $g\notin H_7$.
\end{Example}

\begin{Lemma}
The group $H_{n}$ (resp. $H_{\infty}$) is isomorphic to the wreath product $S_k\wr S_{n}$ (resp. $S_k\wr S_{\infty}$) for every $n\in \N_+$. The cardinality of $H_{n}$ is equal to $n!(k!)^n$.
\end{Lemma}
\begin{proof}
Directly follows from the previous Remark.
\end{proof}

\begin{Definition}
Let $g$ be an arbitrary permutation in $S_{kn}$. The \textit{ordinary $H_{\infty}$-support} $[g]^o_{H}$ and the \textit{(completed) $H_{\infty}$-support} $[g]_{H}$ of $[g]$ are defined as follows:
\[
\text{$[g]^o_{H}:=\{\Gamma_i|g(\Gamma_i)\neq \Gamma_j,\, \forall j\in \N\}, \quad $ and
$\quad [g]_{H}:=\bigcup_{\Gamma_i\in [g]^o_{H}}\Gamma_i.$}
\]
\end{Definition}

The $H$-support $[g]_{H}$ is a finite union $\Gamma$-parts. As a result the group $H_{[g]_{H}}$ makes sense, cf. Eq. \ref{generalized H-A}. Note that the permutation $g$ may act on $([g]_H)^c$ non-trivially.
\begin{Example}
We reconsider the permutation $g=g=(1,8,18,21,6,10,13,2,11,3,12)(4,16,19)(5,17,20)$ in $S_{21}$ defined in Example \ref{example of elements}. The $H$-support of $g$ is
\[
[g]_H=\{\Gamma_1,\Gamma_2,\Gamma_3,\Gamma_4,\Gamma_5\},
\] while $g(\Gamma_6)=\Gamma_7$ and $g(\Gamma_7)=\Gamma_2$. Observe that there is no inclusion relation between the support $\N(g)$ and $H$-support $[g]_H$ of $g$.
\end{Example}

%%%%%%%%%%%%%%%%%%%%%%%%%%%%%%%%%%%%%%%%%%%%%
%%%%%%%%%%%%%%%%%%%%%%%%%%%%%%%%%%%%%%%%%%%%%
%%%%%%%%%%%%%%%%%%%%%%%%%%%%%%%%%%%%%%%%%%%%%
%%%%%%%%%%%%SECTION 3.1 %%%%%%%%%%%% %%%%%%% %%%%%%%% %%%%%%%%%%%%%%%%%%%%%%%%%%%%%%%%%%%%%%%%%%%%%%%%%%%%%%%%%%%%%%%%%%%%%%%%%%%%%%%%%%%%%%%%%%%%%%%%

\section{Hecke algebra of a pair $(G,H)$}

In this section we review the basics of Hecke algebras of attached to finite groups.

Let $G$ be a finite group and $H$ be a subgroup of $G$. Denote the set of $H$-double cosets in $G$ by $\mc{D}$. Then the Hecke algebra $\mc{H}=\mc{H}(G,H)$ of the pair $(G,H)$ is defined to be the $\C$-sub-algebra of $\C[G]$ generated by the elements
\begin{equation}\label{generators of the hecke algebra definition}
S_{\lambda}=\frac{1}{|H|}\sum_{g\in \lambda}g,
\end{equation}
where $\lambda$ runs over the set $\mc{D}_n$ of the $H_n$-double cosets. It turns out that the set $\{S_{\lambda}\mid \lambda\in \mc{D}\}$ is in fact a basis for $\mc{H}$. This means, if  $\lambda$ and $\mu$ are fixed double cosets then there exist a unique complex number $c^{\eta}_{\lambda,\mu}$ for each double coset $\eta$ in $\mc{D}$ such that
\begin{equation}\label{structure constants}
S_{\lambda}\cdot S_{\mu}=\sum_{\eta\in\mathcal{D}}c_{\lambda,\mu}^{\eta}S_{\eta}.
\end{equation}
The constants $c_{\lambda,\mu}^{\eta}$ are called \textit{structure constants} and they determine the multiplicative structure on $\mathcal{H}$ uniquely.

%%%%%%%%%%%%%%%%%%%%%%%%%%%%%%%%%%%%%%%%%%%%%
%%%%%%%%%%%%%%%%%%%%%%%%%%%%%%%%%%%%%%%%%%%%%
%%%%%%%%%%%%%%%%%%%%%%%%%%%%%%%%%%%%%%%%%%%%%
%%%%%%%%%%%%SECTION 3.1 %%%%%%%%%%%% %%%%%%% %%%%%%%% %%%%%%%%%%%%%%%%%%%%%%%%%%%%%%%%%%%%%%%%%%%%%%%%%%%%%%%%%%%%%%%%%%%%%%%%%%%%%%%%%%%%%%%%%%%%%%%%

\subsection{Reverted action and structure constants in terms of centralizers}

In this subsection we show that the calculation of the structure constants can be interpreted in terms of $H$-centralizers of certain elements. We start with fixing three $H$-double cosets $\lambda,\mu,\eta$ in $G$. If $A$ is a subset of$\eta$, then \textit{fiber} of the cartesian product $\lambda\times \mu$ in $A$ is defined by setting
\begin{equation}\label{fiber set definition}
V(\lambda\times\mu:A):=\{(g_1,g_2)\in
\lambda\times\mu\mid g_1g_2\in A\}.
\end{equation}
The next lemma reduces the study of the structure constants to the study of the fibers of the form $V(\lambda\times \mu:g^H)$, where $g^H$ denotes the $H$-conjugacy class of $g$.
\begin{Lemma}\label{first calculation of the structure constants}
If $g$ is an element of the dobule coset $\eta$, then
\begin{eqnarray}
c_{\lambda,\mu}^{\eta} & = &\frac{|V(\lambda\times \mu:g|)}{|H|}=\frac{|V(\lambda\times \mu:g^H)|}{|H||g^H|}=\frac{|V(\lambda\times \mu:g^H)||C_H(g)|}{|H|^2},
\end{eqnarray}
where $C_H(g)$ denotes the centralizer of $g$ in $H$.
\end{Lemma}
\begin{proof}
The first equality directly follows from the definition of the structure constants and the fact that $|V(\lambda\times \mu:z)|$ is equal to the number of times $z$ occurs in the product $(\sum_{g_1\in \lambda}g_1)(\sum_{g_2\in \lambda}g_2)$. The second equality is a consequence of the identity
\[
V(\lambda\times\mu:A\cup B)=V(\lambda\times\mu:B)\cup V(\lambda\times\mu: B),
\]
which holds for disjoint subsets $A,B$ of $\eta$. The final equation is consequence of the formula
\begin{equation}\label{size of the conjugacy class}
|g^H|=\frac{|H|}{|C_H(g)|}.
\end{equation}
\end{proof}

We now consider the reverted action of $H\times H$ on the group $G\times G$, which is defined by the rule
\[
(h_1,h_2)\cdot_{H\times H} (g_1,g_2)=(h_1g_1h_2^{-1},h_2g_2h_1^{-1}),
\]
where $h_1,h_2\in H$ and $g_1,g_2\in G$.
The fiber set $V=V(\lambda\times\mu:g^H)$ is closed under ${H\times H}$ action, where $g$ is an arbitrary element in the double coset $\eta$. It is equal to disjoint union of ${H\times H}$ orbits, say equal to the union of the orbits $\Lambda_1,\cdots,\Lambda_r$. Let  $(g_{1i},g_{2i})$ be the representative of the orbit $L_i$, for $i=1,\dots, r$. Using the orbit formula for group actions, we deduce
\begin{equation}\label{size of the orbits}
|\Lambda_i|=|H\times H \cdot (g_{1i},g_{2i})|=\frac{|H\times H|}{Stab_{H\times H}(g_{1i},g_{2i})}.
\end{equation}
Combining last the equality of Lemma \ref{first calculation of the structure constants} with Eq.\eqref{size of the orbits} and bearing in mind the fact that the union of $\Lambda_i$, $i=1,\dots,r$, is equal to the fiber set $V$ yield the following formula for the structure constants:
\begin{equation} \label{structure constants first}
c_{\lambda,\mu}^{\eta}=\sum_{i=1}^r\frac{|\Lambda_i||C_H(g_{1i}g_{2i})|}{|H|^2}=  
\sum_{i=1}^r \frac{|C_H(g_{1i}g_{2i})|}{{|Stab_{H\times H}(g_{1i},g_{2i}})|}.
\end{equation}

\begin{Remark}\label{stabilizer parts determine each other}
A pair $(h_1,h_2)\in H\times H$ is contained in the stabilizer of a pair $(g_1,g_2)$ if and only if $h_1g_1h_2^{-1}=g_1$ and $h_2g_2h_1^{-1}=g_2$, which is equivalent to $h_2=g_1^{-1}h_1g_1$ and $h_1=g_2^{-1}h_2g_2$. In particular, given $(g_1,g_2)$, the elements $h_1$ and $h_2$ determine each other uniquely  whenever $(h_1,h_2)$ stabilizes $(g_1,g_2)$.
\end{Remark}
\begin{Lemma}\label{size of the stabilizer in terms of an intersection}
The projection map $\pi:(h_1,h_2)\longmapsto h_1$ defines a bijection between $Stab_{H\times H}(g_1,g_2)$ and the intersection $C_H(g_1g_2)\cap g_1Hg_1^{-1}$.
\end{Lemma}
\begin{proof}
Injectivity of $\pi$ is a consequence of Remark \ref{stabilizer parts determine each other}. We prove that $\pi$ is a surjection onto the intersection $C_H(g_1g_2)\cap g_1Hg_1^{-1}$. To show that $\pi$ sends elements into $C_H(g_1g_2)\cap g_1Hg_1^{-1}$, we pick an arbitrary element $(h_1,h_2)$ from the stabilizer of the pair $(g_1,g_2)$. From the previous remark we get $h_1=g_1h_2g_1^{-1}$ and hence $h_1\in g_1Hg_1^{-1}$. Combining the equalities $h_1g_1h_2^{-1}=g_1$ and $h_2g_2h_1^{-1}=g_2$ yields
\[
h_1g_1g_2h_1^{-1}=h_1g_1h_2^{-1}h_2g_2h_1^{-1}=g_1g_2,
\]
which proves that $h_1$ is an element of the centralizer $C_H(g_1g_2)$. Thus $h_1$ is an element of the intersection $C_H(g_1g_2)\cap g_1Hg_1^{-1}$. Conversely, assume that $h_1$ is an element of the intersection $C_H(g_1g_2)\cap g_1Hg_1^{-1}$. Since conjugation is an automorphism, there exists a unique element $h_2\in H$ such that $h_1=g_1h_2g_1^{-1}$. We claim that $(h_1,h_2)$ stabilizes $Stab_{H\times H}(g_1,g_2)$. The equality $h_1=g_1h_2g_1^{-1}$ can be written as $h_1g_1h_2^{-1}=g_1$. So we just need to show that $h_2g_2h_1^{-1}=g_2$. Substituting $h_2=g_1^{-1}h_1g_1$ and using the fact that $h_1$ commutes with $g_1g_2$ we deduce
\[
h_2g_2h_1^{-1}=(g_1^{-1}h_1g_1)g_2h_1^{-1}=g_1^{-1}g_1g_2=g_2.
\] This finishes the proof of the surjectivity.
\end{proof}

\begin{Corollary}\label{structure functions calculation as fractions}
The structure constant $c_{\lambda,\mu}^{\eta}$ is a non-negative integer and given by the formula
\begin{eqnarray}
c_{\lambda,\mu}^{\eta} & = & \sum_{i=1}^r\frac{|C_H(g_{1i}g_{2i})|}{|C_H(g_{1i}g_{2i})\cap g_{1i}Hg_{1i}^{-1}|},\nonumber
\end{eqnarray}
where the elements $g_{1i},g_{2i}$ are as above.
\end{Corollary}
\begin{proof}
Use Lemma \ref{size of the stabilizer in terms of an intersection} and Eq. \eqref{structure constants first}.
\end{proof}

%%%%%%%%%%%%%%%%%%%%%%%%%%%%%%%%%%%%%%%%%%%%%
%%%%%%%%%%%%%%%%%%%%%%%%%%%%%%%%%%%%%%%%%%%%%
%%%%%%%%%%%%%%%%%%%%%%%%%%%%%%%%%%%%%%%%%%%%%
%%%%%%%%FAMILY OF HECKE ALGEBRAS %%%%%%%%%%%% %%%%%%% %%%%%%%% %%%%%%%%%%%%%%%%%%%%%%%%%%%%%%%%%%%%%%%%%%%%%%%%%%%%%%%%%%%%%%%%%%%%%%%%%%%%%%%%%%%%%%%%%%%%%%%%

\subsection{Family of Hecke algebras}

In this subsection we fix two ascending chains $(G_n)_{n\in \N_+}$ and $(H_n)_{n\in \N_+}$  of finite groups such that $H_n\subset G_n$, for every positive integer $n$. The set of $H_n$-double cosets in $G_n$ is denoted by $\mc{D}_n$  The union of the groups $G_n$ (resp. $H_n$) will be denoted by $G_{\infty}$ (resp. $H_{\infty}$). We view the group $H_{\infty}$ as a subgroup of $G_{\infty}$ and denote the $H_{\infty}$-double cosets in $G_{\infty}$ is denoted by $\mc{D}_{\infty}$.

\begin{Definition}
The family $(G_n,H_n)_{n\in \N_+}$ is called \textit{saturated}, if for all positive integers $m\leq n$ and for all $H_n$-double coset $\lambda_n$, the intersection $\lambda_n\cap G_m$
is either empty or equal to a single $H_m$-double coset.
\end{Definition}

Now we consider a saturated family $(G_n,H_n)_{n\in \N_+}$ and fix two positive integers $m\leq n$. We claim that the map
\[
\varphi^n_{m}:H_mgH_m\longmapsto H_ngH_n
\] is an injection from $\mc{D}_m$ into $\mc{D}_n$. Indeed, assume that $H_ngH_n=H_ng'H_n$ for some $g,g'\in G_m$. Since the family $(G_n,H_n)_{n\in \N_+}$ is saturated, the non-empty intersection
\[
(H_ngH_n)\cap G_m=(H_ng'H_n)\cap G_m
\]
is equal to a single $H_m$-double coset. But the intersection contains both $g$ and $g'$, which means that $H_m$-double cosets of $g$ and $g'$ are equal. For a saturated
family, we identify $\mc{D}_m$ with its image in $\mc{D}_{n}$ whenever $m\leq n$. The set $\mathcal{D}$ of $H_{\infty}$-double cosets
can be identified with the union of $\mc{D}_n$. If $\lambda\in \mc{D}$
set $\lambda(n):=\lambda\cap G_n$. As a result of being
saturated, $\lambda(n)$ is either empty or an $H_n$-double coset in $G_n$. Using this observation one can write
$\mathcal{D}_n=\{\lambda(n)\mid\lambda\in \mc{D},\lambda(n)\neq \emptyset\}$. If
$\lambda(n)\neq \emptyset$ set
\begin{equation}
S_{\lambda}(n)=\frac{1}{|H_n|}\sum_{g\in \lambda(n)}g.
\end{equation} The set $\{S_{\lambda}(n)\mid \lambda\in\mc{D},\lambda(n)\neq \emptyset\}$ is a basis of the algebra $\mc{H}(G_n,H_n)$ and thus there exists non-negative integers $c_{\lambda,\mu}^{\eta}(n)$ such that
\[
S_{\lambda}(n)\cdot S_{\mu}(n)=\sum_{\eta\in\mathcal{D}}c_{\lambda,\mu}^{\eta}(n)S_{\eta}(n).
\]
The functions $c_{\lambda,\mu}^{\eta}(n)$ are called the \textit{{structural
coefficients}} of the family $(G_n,H_n)$.

\begin{Definition}
A saturated family $(G_n,H_n)_{n\in \N_+}$ is called \textit{{admissible}} if
for some $g\in\eta$, the fiber set
\[
V(\lambda\times\mu:g^{H_{\infty}})=\{(g_1,g_2)\in \lambda\times\mu:g_1g_2\in g^{H_{\infty}} \}
\]admits finitely many ${H_{\infty}}\times {H_{\infty}}$ orbits.    
\end{Definition}

\begin{Corollary}\label{abstract formula for structure functions}
Let $(G_n,H_n)_{n\in \N}$ be an admissible family and $\lambda$, $\mu$, and $\eta$ be three $H_{\infty}$-double coset. Let
$m$ be the minimal positive integer such that the intersections $\lambda(m)$, $\mu(m)$, and $\eta(m)$ are all non-empty. There exist elements  $g_{11},\dots,g_{1r}$ in $\lambda(m)$, and $g_{21},\dots,g_{2r}$ in $\mu(m)$ such that
and  \begin{eqnarray}
c_{\lambda,\mu}^{\eta}(n) & = & \sum_{i=1}^r\frac{|C_{H_n}(g_{1i}g_{2i}|}{|C_{H_n}(g_{1i}g_{2i})\cap g_{1i}H_{n}g_{1i}^{-1}|}.\nonumber
\end{eqnarray}
\end{Corollary}
\begin{proof}
Follows from the previous discussion and Corollary \ref{structure functions calculation as fractions}.
\end{proof}

%%%%%%%%%%%%%%%DOUBLE%%%%%%%%%%%%%%%%%%%%%%%%%%%%%%
%%%%%%%%%%%%%%%COSETS%%%%%%%%%%%%%%%%%%%%%%%%
%%%%%%%%%%%%%%AND%%%%%%%%%%%%%%%%%%%%%%%%
%%%%%%%%POLYNOMIALITY %%%%%%%%%%%% %%%%%%%
%%%%%%%% %%%%%%OF%STRUCTURAL%FUNCTIONS%%%%%%%%%%%%%%%%%%%%%%%%%%%%%%%%%%%%%%%%%%%%%%%%

\section{$H_n$-double cosets and the polynomiality of the structural coefficients}

In this section we prove that the family $(S_{kn},H_n)_{n\in\N_+}$ is admissible and the structural coefficients are polynomials in $n$. We start with a parametrization the $H_n$-double cosets.

%%%%%%%%%%%%%%%%%%%%%%%%%%%%%%%%%%%%%%%%%%%%%%%%%%%%%%%%%%%
%%%%%%%%%%%%%%%%%%%%%%%%%%%%%%%%%%%%%%%%%%%%%%%%%%%%%%%%%%%
%%%%%%%%%%%%$H_n$-double cosets%%%%%%%%%%%%%%%%%%%%%%%%%%%%
%%%%%%%%%%%%%%%%%%%%%%%%%%%%%%%%%%%%%%%%%%%%%%%%%%%%%%%%%%%
%%%%%%%%%%%%%%%%%%%%%%%%%%%%%%%%%%%%%%%%%%%%%%%%%%%%%%%%%%%
%%%%%%%%%%%%%%%%%%%%%%%%%%%%%%%%%%%%%%%%%%%%%%%%%%%%%%%%%%%
%%%%%%%%%%%%%%%%%%%%%%%%%%%%%%%%%%%%%%%%%%%%%%%%%%%%%%%%%%%

\subsection{$H_n$-double cosets}

We first define a red-blue graph to temporarily parameterize the double-cosets.

\begin{Definition} Let $g$ be a permutation in $S_{kn}$. The red-blue graph $\mathbf{G}_g=(\mathbf{V}_g,\mathbf{E}_g)$ is the unique graph satisfying the following: The vertex set $\mathbf{V}_g$ of the graph $\mathbf{G}_g$ consists of the integer pairs of the form $(i,g(i))$, where $i$ runs over the set $\{1,\dots, kn\}$. The edge set $\mathbf{E}_g$ consists of red edges and blue edges. Let $(r,g(r))$ and $(s,g(s))$ be two vertices in $\mathbf{V}_g$. There is a red edge between $(r,g(r))$ and $(s,g(s))$ if and only if $p(r)=p(s)$; there is a blue edge between $(r,g(r))$ and $(s,g(s))$ if and only if $p(g(r))=p(g(s))$. Notice that, when $k=2$, the graph $\mathbf{G}_g$ is same as the graph defined in \cite[Ch.VII, Sc.2]{Mac}.
\end{Definition}

\begin{Lemma}\label{double cosets are equal if graphs are isomorphic}
The isomorphism class of the red-blue graph $\mathbf{G}_g$ determines $H_ngH_n$ completely.
\end{Lemma}
\begin{proof}
The proof is a direct generalization of \cite[2.1(i), Ch.VII, Sc.2]{Mac}. Let $g,g'$ be two elements of $S_{kn}$ and assume that $\mathbf{G}_g$ and
$\mathbf{G}_{g'}$ are isomorphic. There is a bijection between the vertex sets
\begin{eqnarray}
\mathbf{V}_{g} & \longrightarrow & \mathbf{V}_{g'} \nonumber\\
(r,g(r)) & \longmapsto & (h(r),h'(g(r))\nonumber
\end{eqnarray}
that preserves the red edges and blue
edges. Since the pair $(h(r),h'(g(r))$ is an
edge of $\mathbf{G}_{g'}$, by definition of the vertex set  $\mathbf{V}_{g'}$, it follows that
$g'h(r)=h'g(r)$, for all positive integers $r$ in $[kn]$. Since the red edges
are mapped to red edges, $p(r)=p(s)$
implies $p(h(r))=p(h(s))$, hence
$h\in H_n$. Similarly, blue edges are
permuted among themselves,
$p(g(r))=p(g(s))$ implies
$p(hg(r))=p(hg(s))$. As a consequence the permutation $h'$ is also contained in $H_n$. The equality
$g'h(r)=h'g(r)$ now implies that the elements $g$ and
$g'$ are contained in the same $H_n$-double
coset.

It is straightforward to prove that the elements contained in the same double coset
have isomorphic red-blue graphs.
\end{proof}

\begin{figure}[h]
\begin{tikzpicture}
[scale=.9,auto=left]
\node (n1) at (0,0) {$w_{1,8}$};
\node (n2) at (2,0)  {$w_{2,11}$};
\node (n3) at (1,1)  {$w_{3,12}$};
\node (n7) at (6,0) {$w_{8,18}$};
\node (n8) at (8,0)  {$w_{7,7}$};
\node (n9) at (7,1)  {$w_{9,9}$};
\node (n4) at (3,3) {$w_{4,16}$};
\node (n5) at (5,3)  {$w_{5,17}$};
\node (n6) at (4,2)  {$w_{6,10}$};
\node (n10) at (10,0) {$w_{10,13}$};
\node (n11) at (11,1)  {$w_{11,3}$};
\node (n12) at (12,0)  {$w_{12,1}$};
\node (n13) at (10,4) {$w_{13,2}$};
\node (n14) at (12,4)  {$w_{14,14}$};
\node (n15) at (11,3)  {$w_{15,15}$};

\node (n16) at (14,0) {$w_{16,19}$};
\node (n17) at (15,1)  {$w_{17,20}$};
\node (n18) at (16,0)  {$w_{18,21}$};
\node (n19) at (14,3) {$w_{19,4}$};
\node (n20) at (15,4)  {$w_{20,5}$};
\node (n21) at (16,3)  {$w_{21,6}$};

\foreach \from/\to in {n1/n2,n2/n3,n3/n1, n4/n5,n5/n6,n6/n4, n7/n8,n8/n9,n9/n7, n10/n11,n11/n12,n12/n10, n13/n14,n14/n15,n15/n13,
n16/n17,n17/n18,n18/n16, n19/n20,n20/n21,n21/n19}
\draw[red,line width=1.5pt] (\from) -- (\to);

\draw[blue,line width=1.5pt]  (n1) .. controls (4,-2) .. (n8);
\draw[blue, line width=1.5pt]  (n1) .. controls (2,6) and (6,6) .. (n9);
\draw[blue, line width=1.5pt]  (n8) .. controls (8,1) .. (n9);

\draw[blue, line width=1.5pt]  (n2) .. controls (2,1) .. (n3);
\draw[blue, line width=1.5pt]  (n2) to (n6);
\draw[blue, line width=1.5pt]  (n3) to (n6);

\draw[blue, line width=1.5pt]  (n4) .. controls (4,4) .. (n5);
\draw[blue, line width=1.5pt]  (n4) .. controls (4,5) and (6.5,4) ..  (n7);
\draw[blue, line width=1.5pt]  (n5) to (n7);

\draw[blue, line width=1.5pt]  (n10) to (n15);
\draw[blue, line width=1.5pt]  (n12) .. controls (9,-3) and (9,2) .. (n13);
\draw[blue, line width=1.5pt]  (n12) .. controls (12,1) .. (n11);

\draw[blue, line width=1.5pt]  (n15) .. controls (12,3) .. (n14);
\draw[blue, line width=1.5pt]  (n13) to (n11);
\draw[blue, line width=1.5pt]  (n14) .. controls (13,3) and (14,-3) .. (n10);

\draw[blue,line width=1.5pt]  (n16) .. controls (14,1) .. (n17);
\draw[blue, line width=1.5pt]  (n17) .. controls (16,1) .. (n18);
\draw[blue, line width=1.5pt]  (n18) .. controls (15,-1) .. (n16);

\draw[blue,line width=1.5pt]  (n19) .. controls (14,4) .. (n20);
\draw[blue, line width=1.5pt]  (n20) .. controls (16,4) .. (n21);
\draw[blue, line width=1.5pt]  (n21) .. controls (15,2) .. (n19);

\end{tikzpicture}

\caption{The red-blue graph $\mathbf{G}_g$ when  $g=(1,8,18,21,6,10,13,2,11,3,12)(4,16,19)(5,17,20)$, where $k=3$ and $n=7$.}
\end{figure}
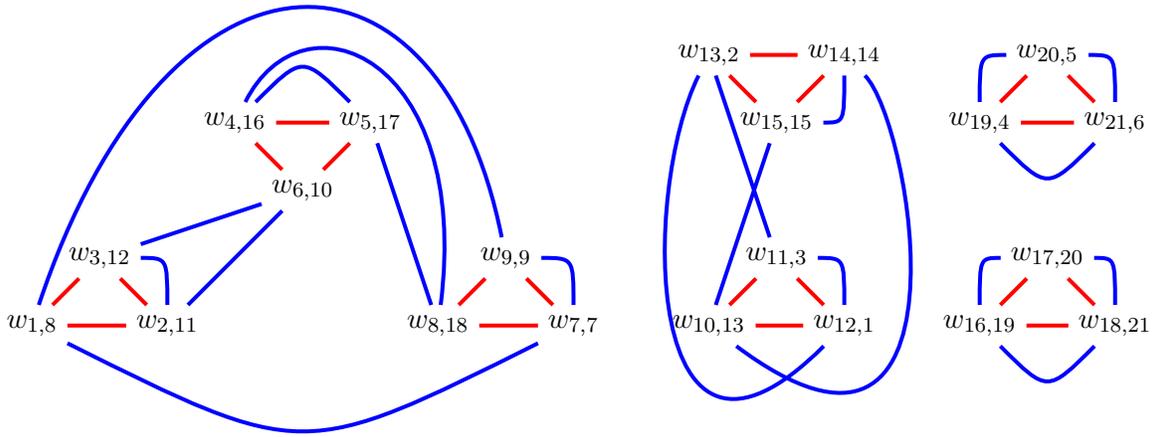

\begin{Remark}
Observe that, the graph obtained from $\mathbf{G}_g$ by erasing the blue edges (resp. red edges) is equal to disjoint union of complete graphs on $k$-vertices.
\end{Remark}

Now we define a second graph, $G_g=(V,E^g)$, whose isomorphism class completely determines the isomorphism class of the red-blue graph $\mathbf{G}_g$. In fact, the graph $G_g$ is obtained from $\mathbf{G}_g$ by collapsing the the $k$-complete red-graphs into a single vertex $v_i$. The graph $G_g$ can be described as follows.

\begin{Definition}\label{graph of an element}
The vertices of the graph ${G_g}=(V,E^g)$ are denoted by $v_1,\cdots,v_n$. For $u=1,\cdots,n$, and $\{r,s\}\subset \Gamma_u$, $r\neq s$, there is an edge $e^g_{r,s}=e^g_{s,r}$:
\[
e^g_{r,s}=e^g_{s,r}=\{\{v_i,r \}\{ v_j,s   \}    \}
\]
where $r\in g(\Gamma_i)$, $s\in g(\Gamma_j)$. Notice that $e^g_{r,s}$ can be written in the following form:
\[
e^g_{r,s}=e^g_{s,r}=\{\{v_{p(g^{-1}(r))},r \},\{ v_{p(g^{-1}(s))},s   \}    \}.
\]
We denote the set of edges $e^g_{r,s}$ where $r,s\in \Gamma_u$ by $E^g_{u}$.
The isomorphism class of $G_g$ is denoted by $\widetilde{G_g}$. The set of isomorphism classes of graphs of the form $G_g$ for $g\in S_{kn}$ is denoted by $\mc{G}_{n,k}$.
\end{Definition}

\begin{Example}
Collapsing the complete $3$-graphs with red edges of the graph presented above into points we obtain the graph $G_g$, where $g=(1,8,18,21,6,10,13,2,11,3,12)(4,16,19)(5,17,20)
$.  

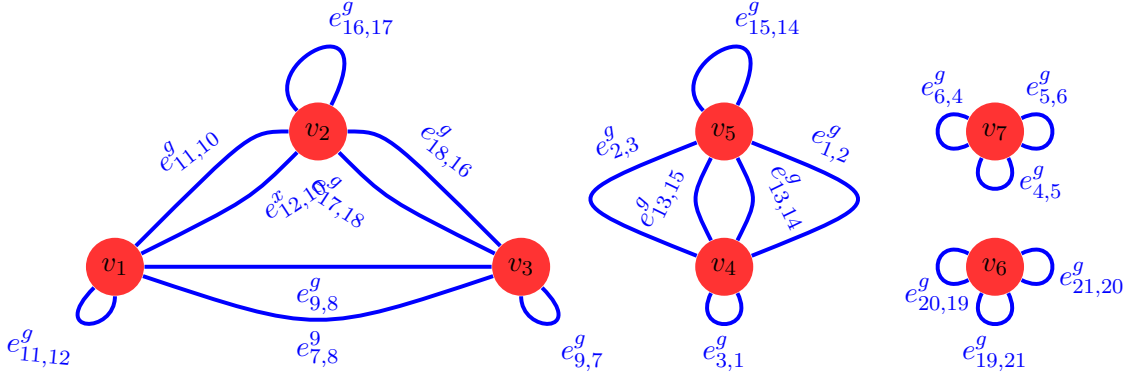
\begin{tikzpicture}
[scale=.9,auto=left]

\node[circle, fill=red!80] (n1) at (0,0) {$v_{1}$};
\node[circle, fill=red!80] (n2) at (3,2)  {$v_{2}$};
\node[circle, fill=red!80] (n3) at (6,0)  {$v_{3}$};
\node[circle, fill=red!80] (n4) at (9,0) {$v_{4}$};
\node[circle, fill=red!80] (n5) at (9,2)  {$v_{5}$};
\node[circle, fill=red!80] (n6) at (13,0) {$v_{6}$};
\node[circle, fill=red!80] (n7) at (13,2)  {$v_{7}$};

\draw[blue,line width=1.5pt] (n1) .. controls (-1,-1) and (0,-1) .. (n1) node[pos=.5,sloped,below left] {$e^{g}_{11,12}$};
\draw[blue,line width=1.5pt] (n2) .. controls (2,3) and (4,4) .. (n2)node[pos=.5,above right] {$e^{g}_{16,17}$};

\draw[blue,line width=1.5pt] (n3) .. controls (6,-1) and (7,-1) .. (n3) node[pos=.5,below right] {$e^{g}_{9,7}$};

\draw[blue,line width=1.5pt] (n1) .. controls (2,2) .. (n2) node[pos=.5,sloped,above left] {$e^{g}_{11,10}$};
\draw[blue,line width=1.5pt] (n1) .. controls (2,1) .. (n2) node[pos=.5,sloped,below right] {$e^{x}_{12,10}$};

\draw[blue,line width=1.5pt] (n2) .. controls (4,2) .. (n3) node[pos=.5,sloped,above right] {$e^{g}_{18,16}$};
\draw[blue,line width=1.5pt] (n2) .. controls (4,1) .. (n3)node[pos=.5,sloped,below left] {$e^{g}_{17,18}$};

\draw[blue,line width=1.5pt] (n1) .. controls (3,0) .. (n3) node[pos=.5,sloped,below] {$e^{g}_{9,8}$};
\draw[blue,line width=1.5pt] (n3) .. controls (3,-1) .. (n1)node[pos=.5,sloped,below] {$e^{9}_{7,8}$};

\draw[blue,line width=1.5pt] (n4) .. controls (9.5,-1) and (8.5,-1) .. (n4)node[pos=.5,sloped,below] {$e^{g}_{3,1}$};
\draw[blue,line width=1.5pt] (n5) .. controls (8,3) and (10,4) .. (n5)node[pos=.5,above right] {$e^{g}_{15,14}$};

\draw[blue,line width=1.5pt] (n4) .. controls (8.5,1) .. (n5)node[pos=.9, sloped,above left  ] {$e^{g}_{13,15}$};
\draw[blue,line width=1.5pt] (n4) .. controls (9.5,1) .. (n5)node[pos=.9, sloped,above right  ] {$e^{g}_{13,14}$};

\draw[blue,line width=1.5pt] (n4) .. controls (6.5,1) .. (n5)node[pos=.9, sloped,above left  ] {$e^{g}_{2,3}$};
\draw[blue,line width=1.5pt] (n4) .. controls (11.5,1) .. (n5)node[pos=.9, sloped,above right  ] {$e^{g}_{1,2}$};

\draw[blue,line width=1.5pt] (n6) .. controls (12.5,-1) and (13.5,-1) .. (n6)node[pos=.5,sloped,below] {$e^{g}_{19,21}$};
\draw[blue,line width=1.5pt] (n6) .. controls (14,-0.5) and (14,0.5) .. (n6)node[pos=.3,right ] {$e^{g}_{21,20}$};
\draw[blue,line width=1.5pt] (n6) .. controls (12,-0.5) and (12,0.5) .. (n6)node[pos=.5,below ] {$e^{g}_{20,19}$};

\draw[blue,line width=1.5pt] (n7) .. controls (12.5,1) and (13.5,1) .. (n7)node[pos=.7,right] {$e^{g}_{4,5}$};
\draw[blue,line width=1.5pt] (n7) .. controls (14,1.5) and (14,2.5) .. (n7)node[pos=.7,above ] {$e^{g}_{5,6}$};
\draw[blue,line width=1.5pt] (n7) .. controls (12,1.5) and (12,2.5) .. (n7)node[pos=.7,above ] {$e^{g}_{6,4}$};

\end{tikzpicture}

\end{Example}

We note that the isomorphism classes of the graph $G_g$ and red-blue graph $\mathbf{G}_g$ are recoverable from each
other uniquely. We also note that the vertices of ${G_g}$ corresponds to the
red-connected components of $\mathbf{G}_g$. Likewise, edges of $G_g$ corresponds to the blue edges of $\mathbf{G}_g$. According to the definition, $v_i$ and $v_j$
are connected by an edge if and only if there exist $r,s$ with $p(r)=p(s)$
such that $g^{-1}(r)\in \Gamma_i$ and $g^{-1}(s)\in \Gamma_j$. From this
observation, one deduces the following:

\begin{Remark}\label{isolated not in the support}
The set $\Gamma_i$ is not element of $[g]^o_{B_n}$ if and only if $\{v_i\}$ is a connected component of $G_g$. Equivalently, the vertex $v_i$ is an isolated vertex of $G_g$ if and only if $g(\Gamma_i)=\Gamma_j$ for some $j\in [n]$; which is tantamount to say that $\Gamma_i$ is contained in the complement $[kn]-[g]_B$.
\end{Remark}

We reformulate Lemma \ref{double cosets are equal if graphs are isomorphic}.

\begin{Corollary}\label{double coset parameterization}
There is a bijection
\begin{eqnarray}
B_n\backslash S_{kn}/B_n & \longrightarrow &  \mc{G}_{n,k} \nonumber\\
B_n g B_n & \longmapsto & \widetilde{G_g}
\end{eqnarray}
for all $n,k> 0$. The graph $G_g$ will be called the \textit{coset type} of $g$.
\end{Corollary}

% If two permutations are contained in the same $B_n$-double coset then their graphs are isomorphic as red-blue graphs.
%\end{Lemma}
%\begin{proof}
% Let $a\in Y_n$ and $b\in G_n$ and $x\in S_{nk}$. We compare the graphs $G_x$ and $G_{abx}$. Between the vertex sets we introduce the following bijection:
% \begin{eqnarray}
% \phi: V_x & \longrightarrow & V_{abx}\nonumber\\
% v_{i,x(i)} & \longmapsto & v_{i,abx(i)}
% \end{eqnarray}
% Now we consider the following bijection between blue edges: A blue edge is determined by the vertices it connects. So, a blue edge is of the form $\{v_{i,x(i)},v_{j,x(j)}\}$ where $p(i)=p(j)$, which means, blue edges are completely determined by the first component of the indices of the vertices. In fact, the map
% \begin{eqnarray}
% \phi_B: E_{B,x} & \longrightarrow & E_{B,abx}\nonumber\\
% \{v_{i,x(i)},v_{j,x(j)}\} & \longmapsto & \{v_{i,abx(i)},v_{j,abx(j)}\}
% \end{eqnarray}
% is a bijection between the set of blue edges of $G_x$ and $G_{abx}$ and it is compatible with the map $\phi$. Hence $\phi,\phi_B$ defines an isomorphism between the blue parts of $G_x$ and $G_{abx}$.
%\end{proof}

\subsection{Minimal double-coset representatives}

\begin{Definition} \label{minimal representative definition}
Let $M\in \mc{D}$ be an $H$-double coset. An element $g\in M$ is called a \textit{minimal representative} if it satisfies the following:
\begin{enumerate}
\item[(a)] The support $\N(g)\subseteq [g]_{H}$ i.e.
\[
g(\Gamma_i)=\Gamma_j \qquad \Rightarrow \qquad i=j \:\:\: \&\:\:\: g(s)=s, \forall s\in \Gamma_i.
\]
\item[(b)] If $p(g(s))=p(s)$ then $g(s)=s$.
\end{enumerate}
\end{Definition}
\begin{Lemma}\label{minimal representative lemma} Every double coset admits a minimal representative. If $g$ is a minimal representative, then $g^{-1}$ is a minimal representative, and $[g]_H=[g^{-1}]_H$. Minimal elements are closed under conjugation by elements of $H$.
\end{Lemma}
\begin{proof} We just prove the existence of minimal representatives as the others are formal consequences of the definition of minimal representative. Let $g$ be a permutation in $S_{kn}$.
\begin{enumerate}
\item[(a)] Assume that $\Gamma_i$ is mapped to $\Gamma_j$ for some positive integers $i,j$ in $[n]$. This means $\Gamma_i$ and $\Gamma_j$ are both contained in the support $\N(g)$ of $g$. Consider the permutation $\tau_{ij}g$. The permutation $\tau_{ij}g$  maps $\Gamma_i$ to itself. Moreover, its support $\N(\tau_{ij}g)$ is contained in the support $\N(g)$. Let $h$ be the unique permutation satisfying $h_{\mid \Gamma_i}=\tau_{ij}g_{\mid \Gamma_i}$, and $h_{\mid \Gamma_i^c}=id$. The permutation $h$ is contained in $H$ and the element $h^{-1}g$ is identity on $\Gamma_i$. By construction $\N(h^{-1}g)\subsetneq \N(g)$. The process terminates in finitely many steps. Resulting element satisfies the property (a) of Definition \ref{minimal representative definition}.
\item[(b)] Next we assume that the permutation $g$ satisfies the property (a). Let $s$ be a positive integer $s\in \N(x)$. We write $g(s)=r$ and assume that $p(r)=p(s)$. The element $(rs)g$ satisfies the property (a) and $\N((rs)g)\subseteq\N(g)-{s}$. This process should terminate in finitely many steps. Resulting element satisfies property (b).
\end{enumerate}
\end{proof}

\begin{Remark}\label{minimal representative left coset}
Let $g$ be a permutation. The above proof shows that one can find a minimal representative in the cosets $Hg$ and $gH$.
\end{Remark}

\begin{Example}
Consider the element $g=(1,8,18,21,6,10,13,2,11,3,12)(4,16,19)(5,17,20)$. Following the procedure presented in the previous proof, we multiply $g$ on the left by the permutation $h=(5,20,17)(4,19,16)(6,21,18)$, and obtain the minimal representative
$(1,8,6,10,13,2,11,3,12)$. We also note that $h=\tau_{2,6}\tau_{7,2}$.
\end{Example}

\subsection{Admissiblity of the family $(S_{kn},H_n)_{n\in \N_+}$}

We denote the natural embedding of $S_{kn}$ into $S_{k(n+t)}$ with $\cdot^{\uparrow t}$. The embedding $\cdot^{\uparrow t}$ maps $H_n$ into $H_{n+t}$. If $g$ is a permutation in $S_{kn}$, then
\begin{equation}\label{effect of raising on the graphs}
G_{g^{\uparrow t}}= G_g \sqcup \underbrace{\circ_{k} \sqcup \cdots \sqcup \circ_{k}}_{\text{$t$ many}}
\end{equation}
where $\circ_{k}$ is the graph with a single vertex with $k(k-1)/2$ loops.
\begin{Lemma}
The family $(S_{kn},H_n)$ is saturated.
\end{Lemma}
\begin{proof}
Let $g_1$ and $g_2$ be two permutations in $S_{kn}$. Assume that the permutations $g_1^{\uparrow t}$ and $g_2^{\uparrow t}$ are contained in same $H$-double coset for some positive integer $t$. By Eq. \eqref{effect of raising on the graphs} we have
\[
G_{g_1^{\uparrow t}}= G_{g_1} \sqcup \underbrace{\circ_{k} \sqcup \cdots \sqcup \circ_{k}}_{\text{$t$ many}},\qquad
G_{g_2^{\uparrow t}}= G_{g_2} \sqcup \underbrace{\circ_{k} \sqcup \cdots \sqcup \circ_{k}}_{\text{$t$ many}}.
\]
By Corollary \ref{double coset parameterization} the graphs $G_{g_1^{\uparrow m}}$ and $G_{g_2^{\uparrow m}}$ are isomorphic. This implies that the graphs $G_{g_1}$ and $G_{g_2}$ are isomorphic. A second use of Corollary \ref{double coset parameterization} finishes the proof.
\end{proof}

\begin{Proposition}\label{admissiblility of the family}
Let $M$, $N,$ $L$ be three $H$-double cosets. If $g$ is a minimal representative for $L$ then the fiber set
\[
V=V(M\times N:g^{H})
\]
admits finitely $H\times H$ orbits.
\end{Proposition}
\begin{proof}
We will show that there exists a positive integer $ku$ such that every orbit has a representative in $V$ contains an element in $S_{ku}\times S_{ku}$. Assume that the pair $(g_1,g_2)$ is contained in $V$. By Remark \ref{minimal representative left coset} there exists a permutation $h_1\in H$ such that $h_1g_1$ is a minimal representative of $M$. The pair
\[
(h_1g_1,g_2h_1^{-1})=(h_1,id)\cdot (g_1,g_2)
\] is contained in the $H\times H$-orbit of $(g_1,g_2)$. On the other hand, by Lemma \ref{minimal representative lemma} we know that the permutation $h_1g_1g_2h_1^{-1}$ is a minimal representative of $L$. Thus, without loss of generality we may assume that $g_1$ and $g_1g_2$ are minimal representatives of their $H$-double cosets. The union
\[
T=[g_1]_{H}\cup[g_1g_2]_{H}
\]
contains the supports $\N(g_1)$ and $\N(g_1g_2)$, because $g_1$ and $g_1g_2$ are both minimal representatives. Since $H$-supports can be covered by the sets of the form $\Gamma_i$, we may assume that the set $T$ is equal to the union of $\Gamma_{i_1},\cdots \Gamma_{i_t}$. The permutation $\tau=\tau_{1i_1}\cdots \tau_{ti_t}$ maps $T$ onto $[kt]$. This means $\tau g_1\tau^{-1}$ and $\tau g_1g_2 \tau^{-1}$ are both elements of $S_{kt}$. Since $\tau g_2\tau^{-1}=\tau g_1^{-1}\tau^{-1}\tau g_1g_2\tau^{-1}$ it follows that $\N(\tau g_2\tau^{-1})\subset [kt]$. As a result we have
\[
(\tau x\tau^{-1},\tau y\tau^{-1})\in V\cap S_{kt}\times S_{kt}.
\]
This means the orbit of $(x,y)$ has a representative in $S_{tk}$. The cardinality of the set $T$ is $kt$ and it is bounded by $ku:=|[g_1]_H|+|[g_1g_2]_H|$. But the cardinalities $[g_1]_H|,|[g_1g_2]_H|$ are constant on double-cosets. Hence, every orbit has a representative in the finite set $S_{ku}\times S_{ku}$. Thus, there exist finitely many orbits.
\end{proof}

\subsection{Centeralizers of the minimal representatives}

In this subsection we investigate the intersection
\[
H_n\cap gH_ng^{-1}
\]
and prove that the structural coefficients $c_{M,N}^{L}(n)$ are polynomials in $n$.

\begin{Remark}\label{centralizer maps the support to itself}
Let $g$ be a permutation in $S_n$ and let $C_{S_n}(g)$ be the centralizer of $g$. The following hold:
\begin{enumerate}
\item If $g'$ is in the centralizer $C_{S_n}(g)$ then $g'$ maps $\N(g)$ to $\N(g)$.
\item The centralizer $C_{S_n}(g)$ admits the following direct-sum product:
\[
C_{S_n}(g)=C_{\N(g)}(g)\oplus S_{[n]-\N(g)}.
\]
\end{enumerate}
\end{Remark}

Our next result is analogous to the first part of Remark \ref{centralizer maps the support to itself}.

\begin{Lemma}\label{normalized part by a minimal representative}
Let $g\in S_{kn}$ be a minimal representative of its $H_n$-double coset. If
\[
h\in H_{n}\cap gH_ng^{-1},
\]
then $h$ maps $[g]$ bijectively onto $[g]$.
\end{Lemma}
\begin{proof}
Let $h,h'$ be two permutations in $H_n$ and assume that $h=gh'g^{-1}$. Assume that $s$ is a positive integer contained in the $H$-support $[g]_B$. We shall show that $h(s)\in [g]_H$. 	To this end we assume that $h(s)\notin [g]_{H}$ and derive a contradiction. By Lemma \ref{minimal representative definition}, the permutation $x^{-1}$ is a minimal representative . So there exists a positive integer $r$ such that $p(r)=p(g^{-1}(s))$ but $p(g(r))\neq p(g(g^{-1}(s)))=p(s)$. We claim that $p(h'(r))\neq p(h'(g^{-1}(s)))$.

We first consider $h'(g^{-1}(s))$. Note that $h'=g^{-1}hg$ and the minimality of $g$ implies $\N(g^{-1})=\N(g)\subset [g]_H$. The assumption $h(s)\notin [g]_{H}$ implies $g^{-1}\big(hg(g^{-1}(s))\big)=g^{-1}(h(s))=h(s)$. In short
\begin{equation}\label{g-1(s)}
h'(g^{-1}(s))=h(s)\notin [g]_B \quad \& \quad
p(h'(g^{-1}(s))=p(h(s)).
\end{equation}
Next we consider $h'(r)$. The two different cases are investigated separately:
\begin{enumerate}
\item First we assume that the integer $hg(r)$ is contained in $[g]_{H}$. This implies that the integer $h'(r)=g^{-1}hg(r)$ is contained in the $H$-support $[g]_{H}$. On the other hand $h(s)\notin [g]_{H}$ as noted in the Eq. \eqref{g-1(s)}. Hence the integers $p(h'(r))$ and $p(h'(g^{-1}(s)))$ can not be equal. This contradictions with the fact that $h'\in H_n$.
\item Next we assume that $hg(r)$ is not contained in $[g]_{H_n}$. The minimality of $g^{-1}$ and the equality $\N(g)=\N(g^{-1})$ together imply that $hg(r)$ is fixed by $g^{-1}$. Consequently we have $h'(r)=g^{-1}hg(r)=hg(r)$. Our initial assumption $p(g(r))\neq p(s)$ yields
\begin{eqnarray}
p(h'(r)) & = & p(hg(r))\nonumber\\
(\text{as $h\in H_n$}) & \neq & p(h(s))\nonumber\\ (\text{Eq. \eqref{g-1(s)}}) & = & p(h'(g^{-1}(s)).
\end{eqnarray} But $h'\in H_n$ and $p(r)=p(g^{-1}(s))$. A contradiction.
\end{enumerate}
\end{proof}

\begin{Corollary}\label{intersection through minimal representative remark}
	If $g$ is a minimal representative contained in $S_{kn}$, then
	\begin{equation}\label{intersection through minimal representative}
	H_{n}\cap gH_ng^{-1}=(H_{[g]_H}\cap gH_{[g]_H}g^{-1})\oplus H_{[kn]-[g]_{H}}.
	\end{equation}
\end{Corollary}

\begin{proof}
	Let $g$ be a minimal representative, and $h_1$ be an element of the intersection $H_n\cap gH_ng^{-1}$. Then there exists $h_2\in H_n$ such that $h_1=gh_2g^{-1}$. Writing $h_2=g^{-1}h_1g$ and applying the previous Lemma we deduce that $h_2$ maps $[g^{-1}]_H$ to itself. Since  $[g^{-1}]_H=[g]_H$, cf. Lemma \ref{minimal representative lemma}, we can write $h_2$ as a product $h_{21}h_{22}$, where $h_{21}\in H_{[g]_H}$ and $h_{22}\in H_{[kn]-[g]_H}$. Consequently
	\begin{eqnarray}
	h_1 = gh_2g^{-1} & = & gh_{21}h_{22}g^{-1}\nonumber\\
	& = & gh_{21}g^{-1}gh_{22}g^{-1}\nonumber\\
	& = & gh_{21}g^{-1}h_{22}
	\end{eqnarray}
	The elements $h_1$ and $h_{22}$ are in $H$, thus $gh_{21}g^{-1}$ is in $H$. As $g$ and $h_{21}$ are permutations of $[g]_H$ it follows that $gh_{21}g^{-1}\in H_{[g]_H}\cap gH_{[g]_H}g^{-1}$. Putting these altogether we see that $h_1$ is an element of the intersection $(H_{[g]_H}\cap gH_{[g]_H}g^{-1})\oplus H_{[kn]-[g]_H}$. The converse inclusion can be shown easily. As a result we get the following decomposition:
	\begin{equation}\label{intersection through minimal representative}
	H_{n}\cap gH_ng^{-1}=(H_{[g]_H}\cap gH_{[g]_H}g^{-1})\oplus H_{[kn]-[g]_{H}},
	\end{equation}
\end{proof}

Now we generalize the second part of Remark \ref{centralizer maps the support to itself}.

\begin{Corollary}
If $g$ is a minimal representative in $S_{kn}$, then
\begin{equation}\label{centralizer of a minimal representative}
C_{H_{n}}(g)=C_{H_{[g]_H}}\oplus H_{[kn]-[g]_{H}}.
\end{equation}
\end{Corollary}
\begin{proof}
Let $h$ be contained in the centralizer $C_{H_n}(g)$. The equality $ghg^{-1}=h$ implies $h\in H_n\cap gH_ng^{-1}$. Now the
result follows by arguing as in the proof of Corollary \ref{intersection through minimal representative remark} and using the fact that $H_{[kn]-[g]_{H}}$ is contained in the
centralizer of $g$.
\end{proof}

In order to prove  the polynomiality of the structure coefficients we need yet another direct-sum decomposition. Recall from the formula given in Corollary \ref{abstract formula for structure functions} that the structural coefficients are finite sums of the quotients of the form
\begin{equation}\label{polynomialty via quotients}
\frac{|C_{H_n}(g_1g_2)|}{|C_{H_n}(g_1g_2)\cap g_1H_{n}g_1^{-1}|}.
\end{equation}
By the proof of Proposition \ref{admissiblility of the family} we may assume that $g_1$ and $g_1g_2$ are minimal representatives of their $H$-double cosets. We will derive the polynomiality result by showing that the quotient in Eq.\eqref{polynomialty via quotients} is a polynomial in $n$. To this end we will combine two direct-sum decompositions given in Eq.\eqref{intersection through minimal representative} and Eq.\eqref{centralizer of a minimal representative}. So lets assume that $g_1$ and $g_2$ are two permutations in $S_{km}$ and $g_1$ and $g_1g_2$ are minimal representatives. Let $T$ denote the union $[g_1g_2]_{H}\cup [g_1]_{H}$ and set
\begin{equation}
\text{$T_1 = [g _1g_2]_{H}- [g_1]_{B},\quad$ $T_2 = [g_1g_2]_{H}\cap [g_1]_{H}, \:\:\:$ and $\quad T_3 = [g_1]_{H} - [g_1g_2]_{H}.$}
\end{equation}  

\begin{Corollary}
For every positive integer $n\geq m$, the intersection $C_{H_n}(g_1g_2)\cap g_1H_{n}g_1^{-1}$ admits the following direct-sum decomposition:
\begin{eqnarray}\label{triple intersection decomposition}
C_{H_n}(g_1g_2)\cap g_1H_{n}g_1^{-1} =  (C_{H_T}(g_1g_2)\cap g_1H_T g_1^{-1}) \oplus H_{[kn]-T}.
\end{eqnarray}
\end{Corollary}
\begin{proof}
Let $h$ be an element of the intersection $C_{H_n}(g_1g_2)\cap g_1H_{n}g_1^{-1}$. Combining Eq.\eqref{triple intersection decomposition} and Eq.\eqref{intersection through minimal representative}  we deduce that
\begin{equation}\label{h preserves T-i's}
h([g_1]_H)=[g_1]_H \qquad h([g_1g_2]_H)=[g_1g_2]_H.
\end{equation} As a result, the permutation $h$ maps $T$ to itself. One can now prove the existence of the direct-sum decomposition arguing as in Corollary \ref{intersection through minimal representative remark} and using using the fact that $h$ fixes $T$ set-wise.
\end{proof}

\begin{Corollary}\label{quotients are polynomials}
If $g_1$ and $g_1g_2$ are minimal representatives of their $H$-double cosets then the quotient
\begin{equation}
\frac{|C_{H_n}(g_1g_2)|}{|C_{H_n}(g_1g_2)\cap g_1H_{n}g_1^{-1}|}.
\end{equation}
is a polynomial in $n$.
\end{Corollary}
\begin{proof}
Using the decomposition given in Eq.\eqref{centralizer of a minimal representative} and Eq.\eqref{triple intersection decomposition} we have the equality
\begin{equation}\label{first equation of quotients are polynomials}
\frac{|C_{H_n}(g_1g_2)|}{|C_{H_n}(g_1g_2)\cap g_1H_{n}g_1^{-1}|}  =  \frac{\mid C_{H_{T_1\cup T_2}}(g_1g_2)\cap g_1H_{T_1\cup T_2} g_1^{-1}\mid }{\mid C_{H_{T_1\cup T_2 \cup T_3}}(g_1g_2)\cap g_1H_{T_1\cup T_2 \cup T_3} g_1^{-1}\mid }\cdot\Big|
\frac{H_{[kn]-(T_1\cup T_2)}}{H_{[kn]-({T_1\cup T_2 \cup T_3})}}\Big|
\end{equation}
If we denote the cardinality $|T_i|$ by $kt_i$ then Eq.\eqref{first equation of quotients are polynomials} can be written as
\begin{eqnarray}
\frac{|C_{H_n}(g_1g_2)|}{|C_{H_n}(g_1g_2)\cap g_1H_{n}g_1^{-1}|}  & = & \zeta_{g_1,g_2}\cdot  \frac{(k!)^{n-{t_1-t_2}}(n-{t_1-t_2})!}{(k!)^{n-{t_1-t_2-t_3}}(n-{t_1-t_2-t_3})!},\nonumber
\end{eqnarray}
where $\zeta_{g_1,g_2}$ denotes the first multiplicand of the RHS of the Eq.\eqref{first equation of quotients are polynomials}, which is independent of $n$. In other words
\begin{equation}
\frac{|C_{H_n}(g_1g_2)|}{|C_{H_n}(g_1g_2)\cap g_1H_{n}g_1^{-1}|}= \zeta_{g_1,g_2}\cdot  (k!)^{t_3}(n-{t_1-t_2})(n-{t_1-t_2}-1)\cdots (n-{t_1-t_2-t_3}+1)\label{polynomiality of the coefficient}.
\end{equation}
\end{proof}

We are ready to show that the structure coefficients are polynomial functions.
\begin{Proposition}
The structure coefficient $c_{M,N}^{L}(n)$ is a polynomial in $n$, for all $H$-double cosets $M,N$, and $L$.
\end{Proposition}
\begin{proof}
Let $M,N$ and $K$ be three $H$-double cosets. Since the family $(S_{kn},B_n)$ is admissible by Proposition \ref{admissiblility of the family}, we may apply Corollary \ref{abstract formula for structure functions}. So the structure coefficient $c_{M,N}^{L}(n)$ is of the following form
\begin{eqnarray}
c_{M,N}^{L}(n) & = & \sum_{i=1}^r\frac{|C_{H_n}(g_{1i}g_{2i})|}{|C_{H_n}(g_{1i}g_{2i})\cap g_{1i}H_ng_{1i}g_{2i}^{-1}|},\nonumber
\end{eqnarray} where $g_{1i}$ and $g_{1i}g_{2i}$ are minimal representatives of their $H$-double cosets.  Since each summand of the RHS is a polynomial in $n$ by Corollary \ref{quotients are polynomials}, so is $c_{M,N}^{L}(n)$.
\end{proof}

\begin{Remark}\label{stability condition remark}
If $[g_1]_{H}\subset [g_1g_2]_{H}$ then $t_3=0$, which implies that the index
\[
\frac{|C_{H_n}(g_1g_2)|}{|C_{H_n}(g_1g_2)\cap g_1H_{n}g_1^{-1}|}
\]
is independent of $n$ and equal to $\zeta_{g_1,g_2}$.
\end{Remark}

\section{The Modified coset type}

We have already seen that the embedding $^{\uparrow t}: S_{kn}  \longrightarrow  S_{kn+kt}$ maps $H_n$ into $H_{n+t}$. If we denote the graph with one vertex and $k(k-1)/2$ loops by $\circ_{k}$, then
\[
G_{g^{\uparrow t}}= G_g \sqcup \underbrace{\circ_{k} \sqcup \cdots \sqcup \circ_{k}}_{\text{$t$ many}}.
\]
This means the attached graph $G_g$ of a finitary permutation $g\in S_{\infty}$ depends on the actual $S_{kn}$ that realizes $g$ as an element of $S_{\infty}$. To remove the dependency we introduce the modified coset type of an element.  Analogous definitions are given in the other works, see \cite[Sec. 3.1]{AC}, and \cite[Sec. 3.1]{WW18}.
\begin{Definition}
Let $g$ be a finitary permutation and assume that $g\in S_{kn}$ for some positive integer $n$. The \textit{modified coset type} of the permutation $g$ is the graph $M_g$ that is obtained from $G_g$ by removing the isolated vertices of $G_g$. The set of isomorphism classes of modified coset types is denoted by $\mc{M}$.
\end{Definition}

From now on a graph $M$ will be denoted as a pair $(V,E)$, where the first (resp. second) component indicates the set of vertices (resp. edges) of $M$. The proof of the next lemma is straightforward.

\begin{Lemma}
\begin{enumerate}
\item Let $g$ be a finitary permutation in $S_{\infty}$. The modified coset type $M_g$ of $g$ is independent of the $S_{kn}$ in which $g$ is realized.
\item The map sending the element $g$ to its modified coset type $M_g$ defines a bijection
\[
H_{\infty}\backslash S_{\infty}/H_{\infty} \longrightarrow \mc{M},
\]
between the $H_{\infty}$-double cosets in $S_{\infty}$ and the set of modified coset types.
\item Let $M=(V,E)$ be a modified type. The group $S_{kn}$ admits an element of $M$ if and only if $|V|\leq n$.
\end{enumerate}
\end{Lemma}

Let $M=(V,E)$ be an arbitrary graph. We will denote the number of connected components of $M$ by $w_M$ or $w(M)$ depending on the convenience. The connected components of $M$ can be denoted by $C_i=(V_i,E_i)$, where $i$ suns over the set $\{1,\dots,w_M\}$. Denote the number of vertices of $C_i$ by $c_i$. Without loss of generality we may assume that $c_i$'s are in non-increasing order. As a result we obtain a partition $(c_1,\cdots,c_{w_M})$, which we denote by $\lambda_{M}$. The weight $||M||$ of the graph $M$ is defined by setting
\begin{eqnarray}
||M|| & := & \sum_{i=1}^{w_M}(c_i-1)=|V|-w_M.
\end{eqnarray}
We note that the modification of an arbitrary graph makes sense and for an arbitrary graph $M$ the equality $||{M}^{\circ}||=||M||$ holds.

With this notation we are able to state the main theorem of this article.

\begin{Theorem}[Stability theorem]\label{stability theorem}
Let $M,N$ and $L$ be three modified types. Then
\[
||M||+||N|| < ||L|| \quad \Longrightarrow \quad c_{M,N}^L(n)=0,
\]
and
\[
||M||+||N|| \geq ||L|| \quad \Longrightarrow \quad c_{M,N}^L(n)\geq 0.
\] In case of the equality $||L||=||M||+||N||$, the polynomial $c_{M,N}^L(n)$ is constant, i.e. independent of $n$.
\end{Theorem}

Let $Z$ be the ring of functions $f:\mathbb{Z}\longrightarrow \Z$ and for each modified coset type $X_M$ be an indeterminate which are algebraically independent over $Z$. Introduce the $Z$-algebra $\mc{F}_{\infty,k}:=Z[X_M:M\in \mc{M}_k]$ where
the multiplication of the indeterminates are defined as
follows.  \[
X_M\cdot X_N=\sum_{\substack{L\in \mc{M} \\ ||L||=||M||+||N||}}c_{M,N}^L(n)X_L.
\]
By the first assertion of the stability theorem the sum is finite and hence make sense. The stability theorem then reads as $\mc{F}_{\infty,k}$ is an associative filtered algebra whose structure coefficients are non-negative integers.

\section{Discrete graph evolution}

Throughout this section we fix an abstract set $V$ of vertices $v_1,\dots,v_n$. Let $M=(V,E)$ be an arbitrary graph. The connected components of $M$ are denoted by $C_i=(V_i,E_i)$, for $i=1,\dots,w_M$. Let $e$ be an edge of $M$. The vertices that the edge $e$ joins are called the \textit{end-points} of the edge $e$. The set of end-points of $e$ is denoted by $V(e)$. The definition of end-points directly generalizes to arbitrary subsets of the edge set $E$. If $D$ is a subset of $E$ then the set of end-points of $D$ is denoted by $V(D)$. The graph on the set of vertices $V(D)$ whose edge set is $D$ is denoted by $G(D)$. If $v$ is a vertex of $M$ then the connected component that contains the vertex $v$ is denoted by $C_M(v)$. In general, if $W$ is a subset of $V$ then the graph $C_M(W)$ is the union
\[
C_M(W)=\bigcup_{v\in W}C_M(v).
\]
The \textit{relative size} $s_M(W)$ of $W$ relative to $M$ is defined as the positive integer
\[
s_M(W)=\#\{i\in \N \mid V_i\cap W\neq \emptyset\}=w(C_M(W)),
\]
the minimal number of connected components of $M$, whose union contains the vertex set $W$.
Let $D$ be a set of edges of the graph $M$ and $G(D)$ denote the sub-graph of $M$ whose vertex set is $V(D)$. The \textit{relative size} $s_M(D)$ of $D$ relative to $M$ is defined as the positive integer
\[
s_M(D)=\#\{i\in \N\mid E_i\cap D\neq \emptyset\}=w(C_M(V(D)))=s_M(V(D)),
\] the relative size the of the set of end-points of $D$. The \textit{non-relative size} $s(D)$ of $D$ is defined as the positive integer
\[
s(D)=w(G(D)),
\]
the number of connected components of the sub-graph $G(D)$.
Note that the non-relative size of an edge set implicitly assumes the existence of an ambient graph. However, this does not provide any problem as the sub-graph $G(D)$ is completely determined by the edge set $D$ alone. We note the the relative size of an edge set is bounded by its non-relative size. That is
\[
s_M(D)\leq s(D).
\] This inequality is a consequence of the fact that vertices that are connected by an edge in $D$ are also connected in the ambient graph $M$. Finally we note that the non-relative site is a sub-additive function in the following sense: If $E_1$ and $E_2$ are two edge sets then
\begin{equation}\label{s is subadditive}
s(E_1\sqcup E_2)\leq s(E_1)+s(E_2),
\end{equation}where the equality does not necessarily hold, even if the edge sets $E_1$ and $E_2$ are disjoint.

\begin{Example}

Consider the graph $M=G_g$ attached to $g=(1,8,18,6,10,13,2,11,3,12)(4,16)(5,17)$, where $k=3$, and $n=7$.

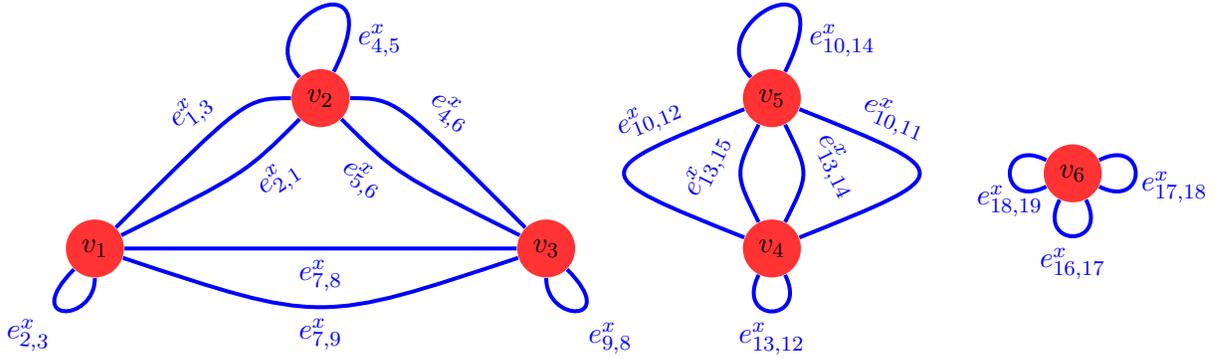
\begin{tikzpicture}

\node[circle, fill=red!80] (n1) at (0,0) {$v_{1}$};
\node[circle, fill=red!80] (n2) at (3,2)  {$v_{2}$};
\node[circle, fill=red!80] (n3) at (6,0)  {$v_{3}$};
\node[circle, fill=red!80] (n4) at (9,0) {$v_{4}$};
\node[circle, fill=red!80] (n5) at (9,2)  {$v_{5}$};
\node[circle, fill=red!80] (n6) at (13,1)  {$v_{6}$};
%\node[draw=none,fill=none] (n7) at (2,3) {$C_1$};
%\node[draw=none,fill=none] (n7) at (8,3) {$C_2$};
%\node[draw=none,fill=none] (n7) at (14,2) {$C_3$};

\draw[blue,line width=1.5pt] (n1) .. controls (-1,-1) and (0,-1) .. (n1) node[pos=.5,sloped,below left] {$e^{x}_{2,3}$};
\draw[blue,line width=1.5pt] (n2) .. controls (2,3) and (4,4) .. (n2)node[pos=.9, right] {$e^{x}_{4,5}$};

\draw[blue,line width=1.5pt] (n3) .. controls (6,-1) and (7,-1) .. (n3) node[pos=.5,below right] {$e^{x}_{9,8}$};

\draw[blue,line width=1.5pt] (n1) .. controls (2,2) .. (n2) node[pos=.5,sloped,above left] {$e^{x}_{1,3}$};
\draw[blue,line width=1.5pt] (n1) .. controls (2,1) .. (n2) node[pos=.5,sloped,below right] {$e^{x}_{2,1}$};

\draw[blue,line width=1.5pt] (n2) .. controls (4,2) .. (n3) node[pos=.5,sloped,above right] {$e^{x}_{4,6}$};
\draw[blue,line width=1.5pt] (n2) .. controls (4,1) .. (n3)node[pos=.5,sloped,below left] {$e^{x}_{5,6}$};

\draw[blue,line width=1.5pt] (n1) .. controls (3,0) .. (n3) node[pos=.5,sloped,below] {$e^{x}_{7,8}$};
\draw[blue,line width=1.5pt] (n3) .. controls (3,-1) .. (n1)node[pos=.5,sloped,below] {$e^{x}_{7,9}$};

\draw[blue,line width=1.5pt] (n4) .. controls (9.5,-1) and (8.5,-1) .. (n4)node[pos=.5,sloped,below] {$e^{x}_{13,12}$};
\draw[blue,line width=1.5pt] (n5) .. controls (8,3) and (10,4) .. (n5)node[pos=.9, right] {$e^{x}_{10,14}$};

\draw[blue,line width=1.5pt] (n4) .. controls (8.5,1) .. (n5)node[pos=.9, sloped,above left  ] {$e^{x}_{13,15}$};
\draw[blue,line width=1.5pt] (n4) .. controls (9.5,1) .. (n5)node[pos=.9, sloped,above right  ] {$e^{x}_{13,14}$};

\draw[blue,line width=1.5pt] (n4) .. controls (6.5,1) .. (n5)node[pos=.9, sloped,above left  ] {$e^{x}_{10,12}$};
\draw[blue,line width=1.5pt] (n4) .. controls (11.5,1) .. (n5)node[pos=.9, sloped,above right  ] {$e^{x}_{10,11}$};

\draw[blue,line width=1.5pt] (n6) .. controls (12.5,0) and (13.5,0) .. (n6)node[pos=.5,sloped,below] {$e^{x}_{16,17}$};
\draw[blue,line width=1.5pt] (n6) .. controls (14,0.5) and (14,1.5) .. (n6)node[pos=.3,right ] {$e^{x}_{17,18}$};
\draw[blue,line width=1.5pt] (n6) .. controls (12,0.5) and (12,1.5) .. (n6)node[pos=.5,below ] {$e^{x}_{18,19}$};

\end{tikzpicture} We denote the connected components of $M$ from left to right by $C_1,C_2$ and $C_3$. If $W=\{v_1,v_2,v_3,v_5\}$ then the graph $C_M(W)$ is the disjoint union of $C_1$ and $C_2$, because $C_M(v_i)=C_1$ for $i=1,2,3$, and $C_M(v_5)=C_2$. This also means that the relative size $s_M(W)$ of $W$ with respect to $M$ is equal to $2$. Let $D$ be the edge set $\{e^g_{7,8}, e^g_{2,3}, e^g_{10,14}\}$. Then the set of end-points of $D$ is equal to $W$, hence $s_M(D)=2$.
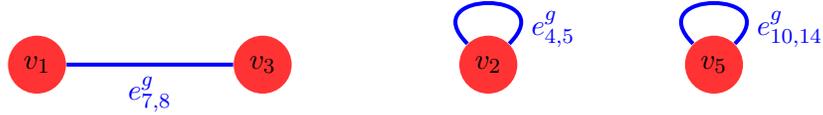
\begin{figure}[h]
\centering

\begin{tikzpicture}

\node[draw=none,fill=none] (n0) at (0,0) {};
\node[circle, fill=red!80] (n1) at (3,0) {$v_{1}$};
\node[circle, fill=red!80] (n2) at (9,0)  {$v_{2}$};
\node[circle, fill=red!80] (n3) at (6,0)  {$v_{3}$};
\node[circle, fill=red!80] (n5) at (12,0)  {$v_{5}$};

\draw[blue,line width=1.5pt] (n2) .. controls (8,1) and (10,1) .. (n2)node[pos=.9, right] {$e^{g}_{4,5}$};

\draw[blue,line width=1.5pt] (n1) .. controls (4.5,0) .. (n3) node[pos=.5,sloped,below] {$e^{g}_{7,8}$};

\draw[blue,line width=1.5pt] (n5) .. controls (11,1) and (13,1) .. (n5)node[pos=.9, right] {$e^{g}_{10,14}$};
\end{tikzpicture}
\caption{The graph $G(D)$ where $D=\{e^g_{7,8}, e^g_{2,3}, e^g_{10,14}\}$.}
\label{fig:my_label}
\end{figure} Number of connected components of $G(D)$ is $3$. As a result, the non-relative size $s(D)$ of the edge set $D$ is $3$, which is an example of the strict inequality $s_M(D)<s(D)$.

\end{Example}

\begin{Definition}\label{edge replacement definition}
Let $M_0=(V,E_0)$ and $M_1=(V,E_1)$ be two graphs on the set of vertices $V$. The graph $M_1$ is \textit{evolved from} $M_0$ \textit{through the edge replacement pairs} $(E_{0i},E_{1i})_{i=1}^t$, if the collection $E_{01},\dots,E_{0t}$ defines a set partition of the vertex set $E_0$ of $M_0$, and the collection $E_{11},\dots,E_{1t}$ defines a set-partition of the edge set $E_1$ of $M_1$ such that
\[
V(E_{0i})=V(E_{1i})
\]
for all $i=1,\dots,t$. The sequence graphs
\[
M_0=G_0,\cdots, G_t=M_1,
\]
is called the \textit{evolution chain}, where the $G_i$ is obtained from $G_{i-1}$ by replacing the edges $E_{0i}$ by $E_{1i}$, for $i=1,\dots,t$. The number $t$ is called the \textit{length} of the evolution.
\end{Definition}

The change in the edge set $E_{G_{i}}$ of the graph $G_i$ as $i$
increases can be depicted as follows:
\begin{eqnarray}
E_{G_{i-1}} & = & E_{11}\sqcup \cdots \sqcup E_{1{(i-1)}}\sqcup E_{0{i}} \sqcup E_{0{(i+1)}} \sqcup \cdots \sqcup E_{0t},\nonumber\\
E_{G_{i}} & = & E_{11}\sqcup \cdots \sqcup E_{1(i-1)}\sqcup E_{1{i}} \sqcup E_{0{(i+1)}} \sqcup \cdots \sqcup E_{0t}.
\end{eqnarray}
\begin{Lemma}\label{weight formula for abstract edge replacement}
Let $M_0$ and $M_1$ be two graphs and assume that $M_1$ is evolved from $M_0$ through the edge replacement pairs $(E_{0i},E_{1i})_{i=1}^t$. The weight of the graphs in the evolution chain obey the inequality
\begin{equation}\label{weight og H-i and H-i+1}
||G_{i}||\leq ||G_{i-1}||+(s(E_{0i})-1).
\end{equation} In particular
\[
||M_1||\leq ||M_0||+\big (\sum_{j=1}^t s(E_{0i}) \big )-t.
\]
\end{Lemma}

\begin{proof}
Let $i$ be fixed positive integer in $[t]$, and denote the connected components of $G_{i-1}$ by $C_j=(V_j,E_j)$, for $j=1,\dots,w$. By $W_i$ we mean the end-point sets $V(E_{0i})$ and $V(E_{1i})$, which are equal by assumption. Some of the connected components of $G_{i-1}$ contains vertices contained in $W_i$. Without loss of generality we may assume that $C_1,\cdots,C_l$ are the connected components that contains an element in $W_i$. In other words, the graph $C_{H_{i-1}}(W_i)$ is equal to union of the graphs $C_1,\dots,C_l$. Hence we deduce
\begin{equation}\label{l is less than s-E-oi}
l=s_{G_{i-1}}(W_i)=s_{G_{i-1}}(E_{0i})\leq s(E_{0i}).
\end{equation}
The graph $G_{i-1}$ can be written as a disjoint union of graphs
\begin{equation}\label{G-i-1 decomposition}
G_{i-1}= C_{G_{i-1}}(W_i) \sqcup C_{l+1}\sqcup \cdots \sqcup C_r.
\end{equation}
Since the end-points of $E_{0i}$ and $E_{1i}$ are all contained in $W_i$, the process of passing from $G_{i-1}$ to $G_i$ only effects the graph $C_{G_{i-1}}(W_i).$ Thus the graph $G_i$ admits the decomposition
\begin{equation}\label{G-i decomposition}
G_i=K_i\sqcup C_{l+1} \sqcup \cdots \sqcup C_r,    
\end{equation}
where $K_i$ is the graph obtained from $C_{G_{i-1}}(W_i)$ by replacing the edges $E_{0i}$ with the edges $E_{1i}$. The decomposition in Eq.\eqref{G-i decomposition} implies
\[
||G_i||=||K_i||+||C_{l+1}||+\cdots+ ||C_t||.
\]
The vertex set $V_{K_i}$ of the graph $K_i$ is equal to the vertex set of $C_{G_{i-1}}(W_i)$. Thus the definition of graph weight yields
\[
\Big|||C_{G_{i-1}}(W_i)||-||K_i||\Big|=\big|w(C_{G_{i-1}}(W_i))-w(K_i)\big|=\big|l-w(K_i)\big|.
\]
Using this equality and combining Eqs. \eqref{G-i-1 decomposition} and \eqref{G-i decomposition} we deduce  
\begin{equation}
\Big|||G_i||-||G_{i-1}||\Big|
= \big|l-w(K_i)\big|
\leq l-1
\leq  s_{E_{0i}}-1\label{s is an upper bound}.
\end{equation}
The first inequality follows because $w(C_{H_{i-1}}(W_i))=1$ and $w(K_i)$ is bounded below by $1$, where the second inequality is a consequence of Eq.\eqref{l is less than s-E-oi}.
\end{proof}

Next Lemma will play a crucial role in the proof of the stability of the family $(S_{kn},H_n)_{n\in \N_+}$.

\begin{Corollary}\label{equality implies connected components merge}
Let $M_0$ and $M_1$ be two graphs and assume that $M_1$ is evolved from $M_0$ through the edge replacement pairs $(E_{0i},E_{1i})_{i=1}^t$. If the inequality
\begin{equation}\label{equality implies connected components merge equation}
||M_1||\leq ||M_0||+\big (\sum_{j=1}^t s(E_{0i}) \big )-t    
\end{equation}
is an equality then the connected components of $M_0$ remain connected in $M_1$.
\end{Corollary}
\begin{proof}
Our proof relies on the proof of Lemma \ref{weight formula for abstract edge replacement}. By Eq.\eqref{G-i-1 decomposition} and Eq.\eqref{G-i decomposition} the connected components $C_{l+1},\dots, C_r$ of the graph $G_{i-1}$ remain connected in $G_i$. So we need to prove that the connected components $C_1,\dots,C_l$ of $G_{i-1}$ remain connected in $G_i$. The vertex set of the graph $K_i$ is equal to the union of the connected components $C_1,\dots,C_l$. So it suffices to prove that $K_i$ is connected.
Assume that \eqref{equality implies connected components merge equation} is an equality. Thus Eq.\eqref{s is an upper bound} is an equality, for every $i=1,\dots,t$. In particular $w(K_i)=1$, for every $i=1,\dots,t$, as a result $K_i$ is connected.
\end{proof}

In the next section we will use Lemma \ref{weight formula for abstract edge replacement} to show that the Hecke algebras $\mc{H}_n$ admits a filtration. In the course of implementing Lemma \ref{weight formula for abstract edge replacement} we will need some more terminology concerning the attached graphs, which we will introduce now. So we turn our attention to the graphs of the form $G_g$ for $g\in S_{kn}$.
Recall that the graph $G_g=(V,E_g)$ is defined as follows. The vertex set is by
definition $V=\{v_1,\cdots,v_n\}$ and
\begin{equation}\label{definition of E-x}
E_g:=\bigsqcup_{u=1}^l E_{g,u}
\end{equation}
where
\begin{equation}\label{definition of E-x-i}
E^g_{u}:=\{e^g_{r,s}=e^g_{s,r}|r,s\in \Gamma_u,\:r\neq s\}.
\end{equation}
The end-points of the edge $e^g_{r,s}=e^g_{s,r}$ are $v_{p(g^{-1}(r))}$ and $v_{p(g^{-1}(s))}$. Thus we may write
\begin{equation}\label{definition of e-x-r-s}
e^g_{r,s}=\{\{v_{p(g^{-1}(r))},r \}, \{v_{p(g^{-1}(s))},s \}\}.
\end{equation}
The edges $e^g_{r,s}$ and $e^g_{r',s'}$ are equal if and only if $\{r,s\}=\{r',s'\}$. Note that $p(g^{-1}(r))=i$ just means that $g^{-1}(r)\in \Gamma_i$. A vertex $v_i$ is an end-point of $e^g_{r,s}$ if and only if
$g^{-1}(r)\in \Gamma_i$ or $g^{-1}(s)\in \Gamma_i$.

\begin{Lemma}\label{vertex set of Gamma-l}
Let $g$ be an arbitrary permutation in $S_{kn}$ and $u$ be an element of $[n]$. The set of end-points $V(E^g_{u})$ of the edge set $E^g_{u}$ is given as follows:
\[
V(E^g_{u})=\{v_i\in V:\Gamma_i\cap g^{-1}(\Gamma_u)\neq \emptyset \}.
\]
More generally, if $J\subset \{1,\dots,n\}$ then
\[
V(\bigcup_{u\in J}E^g_{u})=\{v_i\in V:\Gamma_i\cap g^{-1}(\bigcup_{u\in J}\Gamma_{u})\neq \emptyset \}
. \]
\end{Lemma}
\begin{proof}
Directly follows from the definitions.
\end{proof}

\begin{Lemma}
Let $g$ be an arbitrary permutation in $S_{kn}$ and $C$ be a connected component of $G_g$. Let $J_C$ be the set of indices $u$ in $\{1,\dots,n\}$ such that the vertex $v_u$ is contained in $C$. Then there exists a unique subset $g(J_C)$ of $\{1,\dots,n\}$ such that
\begin{equation}
g(\bigcup_{u\in J_C}\Gamma_{u})=\bigcup_{u'\in g(J_C)}\Gamma_{u'}
\end{equation}
\end{Lemma}
\begin{proof}
We first observe that $J_C=\{u\in [n]\mid v_u\in C\}$. Now let us assume that $r$ is an element of $g(\cup_{u\in J}\Gamma_{u})$. Then $p(g^{-1}(r))\in J$ and hence the vertex $v_{p(g^{-1}(r))}$ is an element of the connected component $C$. We shall show that partners of $r$ also contained in the set $g(\cup_{u\in J}\Gamma_{u})$. Let $s$ be a partner of $r$. Then there is an edge $e^g_{r,s}$ that joins the vertices $v_{p(g^{-1}(r))}$ and $v_{p(g^{-1}(s))}$. Since $C$ is a connected component, the vertex $v_{p(g^{-1}(s))}$ is an element of $J_C$. From this we conclude that $p(g^{-1}(s))\in J$. In other words, $s\in g(\cup_{u\in J_C}\Gamma_{u})$.
\end{proof}

\begin{Lemma}\label{conneced components are mapped to union of parts}
In every $H$-double coset one can find a minimal representative $g$ such that $g(J_C)=J_C$ for every connected component $C$ of the graph $G_g$. Such an element will be referred as an element with \textit{closed connected components}.
\end{Lemma}
\begin{proof}
Let $g$ be a minimal element. There is an element $h$ of $H$ that maps
\begin{equation}
\bigcup_{u'\in g(J_C)}\Gamma_{u'}\longrightarrow \bigcup_{u\in J_C}\Gamma_{u}
\end{equation}
bijectively. Up to multiplication with an element of the Young subgroup $S_{\Gamma_1}\times\cdots\times S_{\Gamma_n}$ one may assume $hg(r)=r$ whenever $p(hg(r))=p(r)$. The element $hg$ is minimal and satisfies the claim of the lemma.
\end{proof}

\begin{Remark}\label{equalities for the closed connected components remark}
Let $g$ be an element with closed connected components. We note that the equalities
\begin{equation}\label{equalities for the closed connected components}
g(\bigcup_{u\in J_C}\Gamma_{u})=\bigcup_{u\in J_C}\Gamma_{u}=g^{-1}(\bigcup_{u\in J_C}\Gamma_{u})
\end{equation}
hold for every conneccted component $C$ of $G_g$.  It is also clear that $|J_C|$ is the number of vertices in $C$, which is a direct consequence of the fact that $J_{C}$ is the set of indices $u$ for which $v_u$ is a vertex in $C$.
\end{Remark}

\begin{Lemma}\label{upper bound for the size of an edge set}
If $J$ is an arbitrary subset of $\{1,2,\dots,n\}$, then
\begin{equation}
s(\bigcup_{u\in J} E^g_{u})\leq  |J| .
\end{equation}
\end{Lemma}
\begin{proof}
Let $v_i,v_j\in V(E^g_{u})$. Then there exists $r,s\in \Gamma_u$ such that $g^{-1}(r)\in \Gamma_i$ and $g^{-1}(s)\in \Gamma_j$. But this means $v_i$ and $v_j$ are joined by the edge $e^g_{r,s}$. This proves that $G(s(E^g_{u}))$ is connected. The general case follows from the sub-additivity of $s(\cdot)$, cf. Eq.\eqref{s is subadditive}.
\end{proof}

\section{Proof of stability}

In this section we prove the stability theorem.

\begin{Theorem*}[Stability theorem]\label{stability theorem}
Let $M,N,L$ be three modified types. Then
\begin{equation}\label{filtration}
||L||>||M||+||N|| \quad \Longrightarrow \quad c_{M,N}^L(n)=0,
\end{equation}
and
\begin{equation}\label{stability}
||L||\leq ||M||+||N|| \quad \Longrightarrow \quad c_{M,N}^L(n)\geq 0.
\end{equation}
In case of the equality $||L||=||M||+||N||$, the polynomial $c_{M,N}^L(n)$ is constant, i.e. independent of $n$.
\end{Theorem*}

Let $g_1,g_2$ be two permutations in $S_{mk}$. Assume that the permutations $g_1$ and $g_1g_2$ are minimal representatives of their
$B$-double cosets. If necessary by replacing $g_1$ with $hg_1$ and $g_2$ with $g_2h^{-1}$ for some suitably chosen $h$ in $H$, we may assume that $g_1$ and $g_1g_2$ are minimal representatives, and $g_1$ has closed connected components. Our aim is to show that $G_{g_1g_2}$ is evolved from $G_{g_2}$. Let $C_1,\cdots,C_t$ be the connected components of $G_{g_1}$. If we denote the number of vertices of $C_i$ by $l_i$, then  
\begin{equation}\label{weight of G-x in terms of n-i}
||G_{g_1}||=\sum_{i=1}^t (l_i-1).
\end{equation} We note that by Remark \ref{equalities for the closed connected components remark} the set equality
\begin{equation}\label{x conneced components are mapped to union of parts}
\bigcup_{u\in J_C} g_1(\Gamma_{u})=\bigcup_{u\in J_C} \Gamma_{u},
\end{equation} holds for every connected component $C$ of $G_{g_1}$, since $g_1$ is an element with closed connected components. For each connected component $C$ of $G_g$ we will define edge sets $E^{g_2}_C$ and $E^{g_1g_2}_C$ that will serve as the edge replacement pairs of the evolution of $M_{g_2}$ into $M_{g_1g_2}$. In fact, let
\begin{equation}
E^{g_2}_C:=\bigsqcup_{u\in J_C}E^{g_2}_{u} \qquad \& \qquad E^{g_1g_2}_C:=\bigsqcup_{u\in J_C}E^{g_1g_2}_{u}
\end{equation}
Since the sets $J_C$ covers $[n]$ as $C$ runs over the connected components of $G_{g_1}$, it follows that
\[
E_{g_2}=\bigsqcup_{i=1}^tE^{g_2}_{C_i} \qquad \& \qquad E_{g_1g_2}=\bigsqcup_{i=1}^t E^{g_1g_2}_{C_i}.
\]
\begin{Lemma}
If $C$ is a connected component of $G_{g_1}$ then the equality
\begin{equation}
V(E^{g_2}_{C})=V(E^{g_1g_2}_{C}).
\end{equation}
between the end-points of the edge sets $E_C^{g_2}$ and $E_C^{g_1g_2}$ holds.
\end{Lemma}
\begin{proof}
By Lemma \ref{vertex set of Gamma-l} and Eq.\eqref{x conneced components are mapped to union of parts} we have
\begin{eqnarray}
V(E^{g_2}_{C}) & = & V(\bigcup_{j\in J_C}E^{g_2}_{j})\nonumber\\
(\text{By Lemma \ref{vertex set of Gamma-l}}) & = & \{v_i\in V: \Gamma_i\cap g_2^{-1}(\bigcup_{u\in J_C}\Gamma_{u})\neq \emptyset \}\nonumber\\
& = &\{v_i\in V: \Gamma_i\cap (g_1g_2)^{-1}g_1(\bigcup_{u\in J_C}(\Gamma_{u})\neq \emptyset \}\nonumber\\
(\text{By Eq.\eqref{x conneced components are mapped to union of parts}})& = & \{v_i\in V: \Gamma_i\cap (g_1g_2)^{-1}(\bigcup_{u\in J_C}\Gamma_{u})\neq \emptyset \}\nonumber\\
(\text{By Lemma \ref{vertex set of Gamma-l}})  & = & V(\bigcup_{j\in J_C}E^{g_1g_2}_{j})\nonumber\\
& = & V(E^{g_1g_2}_{C}).
\end{eqnarray}
\end{proof}

By the Lemma we see that $G_{g_1g_2}$ is evolved from $G_{g_2}$ through the edge replacement pairs $(E^{g_2}_{C_i},E^{g_1g_2}_{C_i})_{i=1}^t$. Since the edge set $E^{g_2}_{C_i}$ is equal to the union of $E^{g_2}_{u}$ where $u\in J_{C_i}$, Lemma \ref{upper bound for the size of an edge set} implies that
\begin{equation}\label{upper bound for E-C-g-2}
s(E^{g_2}_{C_i})\leq |J_{C_i}|=l_i.
\end{equation}
Invoking Lemma \ref{weight formula for abstract edge replacement} we deduce the following:
\begin{eqnarray}
||G_{g_1g_2}|| & \leq & ||G_{g_2}||+\sum_{i=1}^t(s(E^{g_2}_{C_i})-1)\nonumber\\
\text{(By Eq.\eqref{upper bound for E-C-g-2})}& \leq & ||G_{g_2}||+\sum_{i=1}^t(l_i-1)\nonumber\\
(\text{By Eq. \eqref{weight of G-x in terms of n-i}})& = & ||G_{g_2}||+||G_{g_1}||.
\end{eqnarray}
This proves the first statement of the stability theorem which states that $c_{M,N}^L=0$ whenever $||L||$ exceeds the sum $||M||+||N||$.
\begin{Proposition}
Let $g_1,g_2$ be two permutations in $S_{kn}$. Assume that $g_1$ is a minimal representative with closed connected components and $g_1g_2$ is a minimal representative. If the equality
\[
||G_{g_1g_2}||=||G_{g_1}||+||G_{g_2}||
\]
holds then the following hold:
\begin{enumerate}
\item $[g_2]_H\subseteq [g_1g_2]_H$,
\item $\N(g_2)\subseteq [g_1g_2]_H$,
\item $[g_1]_H\subseteq [g_1g_2]_H$.
\end{enumerate}
\end{Proposition}

Before proving the Proposition we note the following. In the light of Remark \ref{stability condition remark}, the inclusion $[g_1]_H\subseteq [g_1g_2]_H$ implies that $c_{M,N}^K(n)$ is independent of $n$, hence finishes the proof of the stability theorem. Note also that $\N(g_1)\subset [g_1]_H$ and $\N(g_1g_2)\subset [g_1g_2]_H$ because $g_1$ and $g_1g_2$ are minimal representatives. However, it is not necessarily true that $\N(g_2)\subset [g_2]$.

\begin{proof}
\begin{enumerate}
\item Assume that for some positive integer $j$  the set $\Gamma_j$ is a subset of $[g_2]_H$. Then by Remark
\ref{isolated not in the support}, the vertex $v_j$ is not an isolated vertex of $G_{g_2}$. If  $C$ denotes the connected component of $v_j$ then
$|C|\geq 2$. By Remark \ref{equality implies connected components
merge}, the set of vertices of the connected component $C$
completely contained in a connected component $C'$ of $G_{g_1g_2}$. As a result
$|C'|\geq 2$. Thus the vertex $v_j$ is not an isolated vertex of $G_{g_1g_2}$. Using Remark
\ref{isolated not in the support} once again, we see that $\Gamma_j\subset [g_1g_2]_H$.
This proves the inclusion $[g_2]_H\subset[g_1g_2]_H$.
\item Next we prove $\N(g_2)\subseteq
[g_1g_2]_H$. Using the previous item it suffices to show that $\N(g_2)-[g_2]_H\subset[g_1g_2]_H$. So let $r$ be an integer contained in the difference $\N(g_2)-[g_2]_H$. For simplicity we write $p(r)=i$, hence we have $r\in \Gamma_i$. Since $H$-supports contain elements that are partners as a whole, it follows that $\Gamma_{i}$ is a subset of the complement $[g_2]^c_H$. By definition of $H$-support $g(\Gamma_i)=\Gamma_j$ for some positive integer $j$. Since $r$ is contained in the complement $[g_1g_2]^c_H$ it follows that $\Gamma_i$ is a subset of the complement $[g_1g_2]^c_H$. By the minimality, the permutation $g_1g_2$ acts on the complement $[g_1g_2]^c_H$ as identity. But this means
\[
\Gamma_i=g_1g_2(\Gamma_i)=g_1(\Gamma_j).
\]
Since $g_1$ is a minimal representative it follows that $i=j$ and $g_1$ acts on $\Gamma_i$ as identity. As a result $g_2$ acts on $\Gamma_i$ as identity, in particular $g_2(r)=r$. But we picked the element $r$ from $\N(g_2)-[g_2]$, a contradiction. So either $\N(g_2)-[g_2]=\emptyset$ or our assumption $r\notin[g_1g_2]_H$ is wrong. Both cases imply that $\N(g_2)\subset [g_1g_2]_H$.

\item It suffices to show that $g_1$ acts on $([g_1g_2]_H)^c$ as identity. But this is clear since $g_2$ and $g_1g_2$ act on $([g_1g_2]_H)^c$ as identity.
\end{enumerate}
\end{proof}

\bibliographystyle{plain}
\bibliography{Stability_of_the_Hecke_algebra_of_wreath_products}

\end{document}